\documentclass[letterpaper,11pt]{article}

\usepackage{ucs}
\usepackage[utf8x]{inputenc}
\usepackage{graphicx}
\usepackage{amsfonts}
\usepackage{dsfont}
\usepackage{amssymb}
\usepackage{amsmath}
\usepackage{amsthm} 
\usepackage{enumerate}
\usepackage{stmaryrd}
\usepackage{fullpage}
\usepackage{ifthen}
\usepackage{subfigure}
\usepackage{epic}
\usepackage{authblk}
\usepackage{textcomp}
\usepackage[small]{caption}
\usepackage{bm}

\renewcommand*{\thefootnote}{\fnsymbol{footnote}}

\usepackage{color}
\usepackage{titlesec}
\graphicspath{{eps/}{pdf/}}

\newcommand{\be}{\begin{equation}}
\newcommand{\ee}{\end{equation}}
\newcommand{\bqs}{\begin{equation*}}
\newcommand{\eqs}{\end{equation*}}

\newcommand{\phipm}{\phi_{\mathrm{pm}}}

\newcommand{\rmO}{\mathcal{O}}
\newcommand{\eps}{\varepsilon}
\newcommand{\e}{d_1}

\numberwithin{equation}{section}

  {\begin{trivlist}\item[]{\bf Hypothesis #1 }\em}{\end{trivlist}}

\theoremstyle{plain}
\newtheorem{theorem}{Theorem}[section]
\newtheorem{proposition}[theorem]{Proposition}
\newtheorem{lemma}[theorem]{Lemma}
\newtheorem{corollary}[theorem]{Corollary}

\newtheorem{rmk}[theorem]{Remark}

\newcommand{\clin}{c_{\mathrm{lin}}}
\newcommand{\cps}{c_\mathrm{ps}}
\newcommand{\cpm}{c_\mathrm{pm}}
\newcommand{\nulin}{\nu_\mathrm{lin}}

\newcommand{\R}{\mathbb{R}}
\newcommand{\C}{\mathbb{C}}
\newcommand{\etalin}{\eta_\mathrm{lin}}
\newcommand{\mcl}{\mathcal{L}}
\newcommand{\upl}{u_\mathrm{pl}}
\newcommand{\mca}{\mathcal{A}}
\renewcommand{\Re}{\mathrm{Re} \,}
\newcommand{\ups}{u_\mathrm{ps}}
\newcommand{\etaps}{\eta_\mathrm{ps}}
\newcommand{\vps}{v_\mathrm{ps}}
\newcommand{\utr}{u_\mathrm{tr}}
\newcommand{\Atr}{\mathcal{A}_\mathrm{tr}}
\newcommand{\Ltr}{\mathcal{L}_\mathrm{tr}}

\renewcommand{\u}{\mathbf{u}}
\newcommand{\vpl}{v_\mathrm{pl}}

\newcommand\blfootnote[1]{%
	\begingroup
	\renewcommand\thefootnote{}\footnote{#1}%
	\addtocounter{footnote}{-1}%
	\endgroup
}

 {\begin{trivlist}\item[]\toextbf{Proof#1 }}%
 {\hspace*{\fill}$\rule{0.3\baselineskip}{0.35\baselineskip}$\end{trivlist}}

\begin{document}

\begin{center}
{\fontsize{15}{15}\fontseries{b}\selectfont{Pushed and pulled fronts in a logistic Keller-Segel model with chemorepulsion\blfootnote{The authors acknowledge support through grants NSF DMS-2205663 (A.S.), NSF-DMS-2202714 (M.A.), NSF-DMS-2007759 (M.H.)} }}\\[0.2in]
Montie Avery$^1$, Matt Holzer$^2$, and Arnd Scheel$^3$ \\[0.1in]
\textit{\footnotesize 
$^1$ Boston University, Department of Mathematics and Statistics, Boston, MA, USA \\
$^2$Department of Mathematical Sciences, George Mason University, Fairfax, VA, USA\\
$^3$University of Minnesota, School of Mathematics,   206 Church St. S.E., Minneapolis, MN 55455, USA
}
\end{center}

\begin{abstract}
    We analyze spatial spreading in a population model with logistic growth and chemorepulsion. In a parameter range of  short-range chemo-diffusion, we use geometric singular perturbation theory and functional-analytic farfield-core decompositions to identify spreading speeds with marginally stable front profiles. In particular, we identify a sharp boundary between between linearly determined, pulled propagation, and  nonlinearly determined, pushed propagation, induced by the chemorepulsion. The results are motivated by recent work on singular limits in this regime using PDE methods \cite{GrietteHendersonTuranova}.
\end{abstract}



\section{Introduction}\label{sec:ExistenceGSPT}
We are interested in the Keller-Segel model for chemotactic motion with logistic population growth
\begin{align}
    u_t &= u_{xx} + \chi(u v_x)_x + u(1-u) \\
    0 &= \sigma v_{xx} + u - v.
\end{align}
Here, a population of agents with density $u$ is modeled by a logistic growth term and spatial diffusion, similar to the classical Fisher-KPP equation. In addition, agents in the population $u$ generate a diffusing and decaying chemical $v$. Changes in the density of $v$ are assumed to happen at a much faster scale than changes in the density $u$ so that a term $\tau v_t$ in the second equation is neglected. Crucially, the population $u$ senses gradients of $v$ which then induce a chemotactic motion with sensitivity $\chi$. We will assume throughout that the chemotactic motion is against the gradient, $\chi>0$, so that the release of the chemical induces the population to avoid clustering, while $\chi<0$ leads to the formation of clusters. The case $\chi>0$ in fact stabilizes a constant distribution $u\equiv1$ and the arguably most interesting effects of chemotactic motion arise during the growth of populations. The question we attempt to address here is whether spatial-temporal growth is enhanced by the chemotactic term, that is, if a population that is initially supported in a compact region of $x\in\R$ will spread with speed $c=2$, as is observed in the absence of chemotactic effects $\chi=0$, or if the repulsion introduced by $\chi$ can accelerate spatial spreading and lead to $c>2$. 

To be more precise, observe that absence of agents and chemical, $u = v = 0$ is a linearly unstable equilibrium state. Spatial growth and spreading of disturbances is well described by the propagation of an \emph{invasion front} which leaves behind a new, selected state in its wake, $u=v=1$. We will focus on the regime $\sigma/\chi \ll 1$. This limit has recently been studied in \cite{GrietteHendersonTuranova}, where estimates on the minimal speed for which positive traveling waves exist are derived. Inspired by this analysis, we present here a conceptual approach that relies on dynamical systems and functional analytic methods rather than PDE tools to precisely characterize   speeds in this limiting regime.  Before proceeding, we also note that results regarding invasion fronts and spreading speeds in the chemoattractive case ($\chi<0$) have been studied by several authors; see for example \cite{henderson22,nadin08,salako19}.

We now define $\tilde{x} = \frac{1}{\sqrt{\chi}} x$ and the rescaled unknowns $\tilde{u}(\tilde{x}, t) = u(x,t), \tilde{v}(\tilde{x}, t) = v(x,t)$. Dropping the tildes, we find the rescaled equation
\begin{eqnarray}
u_t &=& \e u_{xx}+ (uv_x)_x +u(1-u) \nonumber \\
0 &=& \delta^2 v_{xx}+u-v,  \label{eq:main} 
\end{eqnarray}
where $d_1 = \frac{1}{\chi}>0$ and $\delta^2 = \frac{\sigma}{\chi}\ll 1$. 

\paragraph{Front selection criteria.} Fixing $d_1$ and $\delta$, the system \eqref{eq:main} admits  traveling front solutions $u(x,t) = U(x-ct), v(x,t) = V(x-ct)$ connecting the stable state $(U,V) = (1,1)$ at $-\infty$ to the unstable state $(U,V) = (0,0)$ at $+\infty$ for many different speeds $c$. This is typical for fronts connecting to unstable states, and the presence of a one-parameter family of invasion fronts (after modding out the spatial translation symmetry) may be predicted by comparing the dimension of the unstable manifold of $(1,1)$ with the stable manifold of $(0,0)$ in the resulting traveling wave equation. To identify which of these fronts is observed when initial conditions are strongly localized, that is, supported for instance on a half line, we therefore need an additional selection criterion. 


A heuristic often used in the mathematics literature and supported by many results in order-preserving systems is to predict propagation at the minimal speed for which \eqref{eq:main} admits strictly positive traveling front solutions. In the absence of a comparison principle, this ``marginal positivity'' is replaced by the 
%
\emph{marginal stability criterion} \cite{vanSaarloosReview,CAMS}:
\[
 \text{\emph{Marginally stable fronts are selected by localized initial conditions.}}
\]
Here, marginal stability may be encoded as marginal spectral stability --- that is, the linearization about the front has spectrum which touches the imaginary axis but is otherwise stable --- in an appropriately chosen weighted norm. The universal validity of this criterion, in particular beyond equations with comparison principles, has recently been established rigorously in \cite{CAMS, Selectionsys}. 
Marginally stable spectrum may be point or essential spectrum. The latter is determined by the linearization in the leading edge, leading to ``linearly determined'' \emph{pulled fronts}. The former depends on the precise shape of the front interface, thus on the precise shape of the nonlinearity, and the associated fronts are commonly referred to as \emph{pushed fronts}. 


\paragraph{Main results.} Traveling front solutions $(u(x,t), v(x,t)) = (U(x-ct), V(x-ct))$ to \eqref{eq:main} solve the traveling wave system
\begin{align}
    0 &= d_1 U_{xx} + c U_x + (U V_x)_x + U - U^2, \nonumber \\
    0 &= \delta^2 V_{xx} + U - V. \label{e: tw}
\end{align}
When $\delta = 0$, we may substitute $V = U$ into the first equation, and find the corresponding \emph{porous medium limit},
\begin{align}
    0 = d_1 U_{xx} + c U_x + (U U_x)_x + U - U^2, \label{e: pme}
\end{align}
where explicit fronts can be found through a transformation to a Nagumo equation; see \cite{GrietteHendersonTuranova} and Appendix~\ref{a: pme limit}. In fact, there exists a precise characterization of spreading speeds in this limit. 
\begin{lemma}\label{lem:PME}\cite{kawasaki17} Selected fronts in \eqref{e: pme} are positive and have speed $c = \cpm(d_1)$, where
\begin{align}
    \cpm(d_1) = \begin{cases}
    \frac{1}{\sqrt{2}} + \sqrt{2} d_1, & d_1 < \frac{1}{2}, \\
    2 \sqrt{d_1}, & d_1 \geq \frac{1}{2}. 
    \end{cases}
\end{align}
With $d_1$ fixed and $c = \cpm(d_1)$, \eqref{e: pme} has a unique heteroclinic orbit $\phipm$ connecting $U = 1$ to $U = 0$. When $d_1 < \frac{1}{2}$, the heteroclinic corresponds to a marginally stable pushed front solution, while for $d_1 > \frac{1}{2}$, it corresponds to a marginally stable pulled front. 
\end{lemma}
Our main results continues this pushed-pulled dichotomy to finite $\delta$.


\begin{theorem}\label{t: main}
Consider \eqref{eq:main} with  $d_1 > 0$, $d_1+1/d_1\leq C$ bounded and $0<\delta<\bar{\delta}(C)$ sufficiently small. There exists a smooth function $d_1^*(\delta)$, defined for $0 \leq \delta < \overline{\delta}$, such that for $\delta\in[0,\bar{\delta})$, we have:
\begin{itemize}
    \item pulled fronts when $d_1 > d_1^*(\delta)$, propagating with the \emph{linear spreading speed} $\clin := 2 \sqrt{d_1}$;
    \item pushed fronts when $d_1 < d_1^*(\delta)$, propagating with the \emph{pushed speed} $c_\mathrm{ps} = \cpm(d_1)+ c_\mathrm{ps,2}\delta^2+\rmO(\delta^4)$, where
    \begin{align} \label{e: cps2}
        c_\mathrm{ps,2}= 
        -\frac{(6 d_1+3) \left(-18 (d_1+1)^{2 d_1+3} d_1^{-2d_1}+(2 d_1 (71 d_1+134)+149) d_1+23\right)}{12 \sqrt{2} (d_1+1) (2 d_1+1)^2}
    \end{align}
    \item a generic pushed-to-pulled transition in the sense of \cite{avery22} at $d_1 = d_1^*(\delta)=\frac{1}{2}+d_{1,2}\delta^2+\mathrm{O}(\delta^4)$, and
    \begin{align} \label{e:d12}
       d_{1,2} = \frac{1}{16} \left(268-243 \log(3)\right) \approx 0.0648259.
    \end{align}
    \item  selected fronts, pushed or pulled, are continuous in $C^\infty_\mathrm{loc}$ as functions of $d_1$ and $\delta\geq 0$.
\end{itemize}
\end{theorem}
The analysis in \cite{avery22} implies that as a function of $d_1$, the pushed speed exhibits a quadratic correction to the linear speed near $d_1^*(\delta)$ and the quadratic coefficient is continuous in $\delta$.

\begin{rmk}
    The fronts constructed in Theorem~\ref{t: main} are positive.  Positivity of the front in the leading edge, $x \gg 1$, and in the wake, $x \ll -1$, are established in the proof of Theorem~\ref{t: main}. Positivity in the intermediate regime can then be established using a maximum principle argument as follows.  Since the fronts are also positive at $\delta = 0$ by Lemma \ref{lem:PME}, if $U(x;\delta)$ were to not be strictly positive for some $\delta > 0$, there would have to be an $x_0> 0$ and $\delta_0 > 0$ such that $U(x_0; \delta_0) = U'(x_0; \delta_0) = 0$ and $U''(x_0; \delta_0) > 0$. The existence of such a point is not possible by inspection of the first equation in \eqref{e: tw}. A similar argument holds for the $V$ component.
\end{rmk}

\paragraph{Overview.} Our approach follows the conceptual approach taken in \cite{CAMS,avery22}: 
\begin{itemize}
    \item[(i)] establish existence of fronts;
    \item[(ii)] identify speeds $c$ with linear marginal stability;
    \item[(iii)] establish selection. 
\end{itemize}
For steps (i) and (ii), a variety of techniques, using ODE, PDE, or topological tools are available. Step (iii) relies on the validity of the marginal stability conjecture and has been established for large classes of parabolic equations \cite{CAMS, Selectionsys}. Our result here covers steps (i) and (ii). We believe that the methods in \cite{CAMS, Selectionsys} can be adapted to establish (iii) in the present situation. 

In somewhat more detail, continuing pushed front solutions is typically amenable to classical perturbation theory approaches. Continuing pulled fronts is more difficult, since the linear spreading speed is associated with a Jordan block for the linearization at the unstable state in the traveling wave equation, and one must carefully study the convergence in this generalized eigenspace. Persistence of pulled fronts may then be established using either geometric desingularization to split up the double eigenvalue, or far-field/core decompositions which explicitly capture asymptotics in the leading edge \cite{AveryGarenaux, CAMS, avery22, Selectionsys}. 

The difficulty in proving Theorem \ref{t: main} is that the perturbation from $\delta = 0$ in \eqref{e: tw} is singular. From the point of view of PDE techniques for constructing traveling waves, for instance by compactness arguments, this poses a substantial technical difficulty in obtaining uniform regularity estimates on $V$; see \cite{GrietteHendersonTuranova} for further discussion of associated technical difficulties with this approach. On the other hand, when viewing the traveling wave system \eqref{e: tw sys} as a dynamical system in the variable $x$, Fenichel's geometric singular perturbation theory \cite{fenichel79} provides a powerful tool for analyzing the singular perturbation. Indeed, using Fenichel's methods we are able to reduce the singularly perturbed traveling wave problem \eqref{e: tw} to the regularized, scalar semilinear problem
\begin{align}
    0 = d_1 u_{xx} + c u_x + f(u, u_x; d_1, \delta), \label{e: reduced front existence}
\end{align}
where $f$ is $C^k$ in all arguments for any fixed $k <\infty$. We can then study persistence of invasion fronts, step (i) above, using the tools developed in \cite{CAMS, avery22} for regular perturbations. Marginal stability, step (ii) above, again uses Fenichel's reduction to regularize the eigenvalue problem, which we can then study using methods of \cite{CAMS, avery22}, although the form of the reduced eigenvalue problem is not quite as simple as \eqref{e: reduced front existence}. 


\paragraph{Outline.}
We use Fenichel's reduction to regularize the singular perturbation in existence and eigenvalue problem in Section \ref{s: regularization}. In Section \ref{s: existence}, we study the resulting regularized traveling wave problem using functional-analytic methods, using methods developed in \cite{AveryGarenaux, CAMS, avery22} to find pulled and pushed front profiles as well as the transition curve.  Section \ref{s: stability} establishes marginal spectral stability of these fronts thus justifying the pushed and pulled terminology. In Section~\ref{sec:numerics}, we briefly compare the expansions obtained in Theorem~\ref{t: main} to those obtained using numerical continuation.  The appendix contains the construction and properties of traveling fronts at $\delta=0$.

\section{Regularization via geometric singular perturbation theory}\label{s: regularization}

\subsection{Reduction of existence problem}
We express \eqref{e: tw} as a dynamical system in the variable $x$ by choosing coordinates $U, W= U', H = \frac{V-U}{\delta^2}, Z = \delta H'$, obtaining
\begin{align}
    U' &= W \nonumber \\
    W' &= - \frac{1}{d_1} \Gamma(U, W, H, Z) \nonumber \\
    \delta H' &= Z \nonumber \\
    \delta Z' &= H + \frac{1}{d_1} \Gamma(U, W, H, Z), \label{e: tw sys}
\end{align}
where
\begin{align}
    \Gamma(U, W, H, Z) = cW + W^2 + \delta W Z + U H + U (1-U). \label{e: Gamma def}
\end{align}
When $\delta=0$ the system (\ref{e: tw sys}) reduces to two algebraic equations coupled to two differential equations. One identifies the following  reduced slow manifold comprised of solutions of the  algebraic of equations in the singular limit $\delta=0$,
\[ \mathcal{M}_0 = \left\{ \left(U,W,H,Z\right) \ |  \ Z=0, \ H=-\frac{cW+W^2+U-U^2}{d_1+U} \right\} .\]
The linearization of (\ref{e: tw sys}) at any such fixed point has two zero eigenvalues and two hyperbolic eigenvalues $\pm \sqrt{1+\frac{U}{\e}}$ for $U\geq 0$.  The eigenspaces of the non-zero eigenvalues are traverse to $\mathcal{M}_0$ and therefore the reduced manifold is normally hyperbolic. Fenichel's Persistence Theorem \cite{fenichel79} implies that $\mathcal{M}_0$ persists as an invariant manifold $\mathcal{M}_{\delta}$ with the following properties.

\begin{proposition}[Reduction for existence problem]\label{p: existence reduction} 
 Fix $0 < \underline{d} < \overline{d}$, $M > 1$, and an integer $k \geq 2$.  There exists a $\overline{\delta}>0$ such that all trajectories of \eqref{e: tw sys} with $|\delta| < \overline{\delta}$, $\underline{d} < d_1 < \overline{d}$ satisfying $-\frac{d_1}{2}<U<M$ and $|W| \leq M$ lie in a slow manifold  $\mathcal{M}_\delta$, which is normally hyperbolic and invariant under the flow of \eqref{e: tw sys}, and may be written as a graph $\mathcal{M}_\delta = \{ (U,W,H,Z) : H = \psi_H (U,W; \delta), Z = \psi_Z(U,W;\delta)\}$, where
\begin{align}
    H &= \psi_H (U, W; \delta) = \psi_H^0 (U, W) + \delta \psi_H^1 (U, W) + \delta^2 \psi_H^2(U,W) + \mathrm{O}(\delta^4),  \nonumber \\
    Z &= \psi_Z (U, W; \delta) = \delta \psi_Z^1 (U, W) + \mathrm{O}(\delta^2), 
\end{align}
and $\psi_{Z/H}$ are $C^k$ in all arguments. Hence, for all such trajectories, $U$ and $W$ solve the reduced system
\begin{align}
    U' &= W \nonumber \\
    W' &= -\frac{1}{d_1} \Gamma (U, W, \psi_H (U, W; \delta), \psi_Z (U, W; \delta)). \label{e: reduced tw sys}
\end{align}
Moreover, $\psi_Z (0, 0; \delta) = \psi_H (0, 0; \delta) = \psi_H (1, 0; \delta) = 0$ for all $|\delta| < \overline{\delta}$.
\end{proposition}
Note that $\psi_{Z/H}$ also depend on $d_1$, but we suppress this dependence in our notation for now. Throughout the rest of the paper, $k$ refers to the fixed $k \geq 2$ in the statement of Proposition \ref{p: existence reduction}. 

Setting $u = U$, we find from \eqref{e: reduced tw sys} the reduced scalar equation
\begin{align}
     0 = d_1 u_{xx} + c u_x + f(u, u_x; \delta, d_1), \label{e: reduced tw}
\end{align}
where
\begin{align}
    f(u, u_x; \delta, d_1) = u_x^2 + \delta u_x \psi_Z (u, u_x; \delta) + u \psi_H (u, u_x; \delta) + u - u^2 \label{e: reduced f def}
\end{align}
is $C^k$ in all arguments. 

\begin{lemma}\label{l: psi expansions}
    Let $(U(x; \delta), W(x; \delta))$ be a solution to the reduced system \eqref{e: reduced tw sys}. We then have
    \begin{align}
        \psi_H^0 (U(\cdot; 0),W(\cdot; 0)) &= U''(\cdot; 0), \\
        \psi_H^1 (U(\cdot; 0),W(\cdot; 0)) &= 0, \label{e: psi h1 expansion}\\
        \psi_Z^1 (U(\cdot; 0),W(\cdot; 0)) &= U'''(\cdot; 0), \\
        \psi_H^2 (U(\cdot; 0),W(\cdot; 0)) &= \frac{U''''(\cdot; 0)-\frac{1}{d_1}U'(\cdot; 0)U'''(\cdot; 0)}{1+ \frac{U(\cdot; 0)}{d_1}}.  \label{e: psi h2 expansion}
    \end{align}
\end{lemma}
\begin{proof}
    Substituting $Z = \psi_Z(U,W; \delta)$ and $H = \psi_H(U,W; \delta)$ into the last equation of \eqref{e: tw sys}, we find
    \begin{align}
        \delta \partial_x \psi_Z (U,W; \delta) = \psi_H (U,W; \delta) + \frac{1}{d_1} \Gamma(U, W, \psi_H(U,W; \delta), \psi_Z (U,W; \delta)) 
    \end{align}
    Along trajectories of the reduced system \eqref{e: reduced tw sys}, the last term is equal to $-U''$, from which we conclude
    \begin{align}
        \psi_H(U,W; \delta) = U'' + \delta \partial_x \psi_Z (U, W; \delta). \label{e: psi H expression}
    \end{align}
    Setting $\delta = 0$, we find $\psi_H(U(\cdot; 0),W(\cdot; 0); 0) = U''(\cdot; 0)$. Expanding the third equation of \eqref{e: tw sys} to leading order, we find $\psi_Z^1(U(\cdot; 0),W(\cdot; 0); 0) = \partial_x \psi_H(U(\cdot; 0),W(\cdot; 0); 0) = U'''$. We then find \eqref{e: psi h1 expansion} and \eqref{e: psi h2 expansion} by expanding \eqref{e: psi H expression} up to order $\delta^2$. 
\end{proof}

\begin{rmk}
    Notice that the reduced equation \eqref{e: reduced tw} is semilinear, with $f$ only depending on $u$ and $u_x$. However, in Lemma \ref{l: psi expansions} we express $\psi_H^0(U, W; 0)$ as $U''$, which makes the resulting reduced equation appear quasilinear. However, since $u$ solves a second order ODE, we can express $U''$ as a function of $U$ and $W$, obtaining for $\delta = 0$
    \begin{align}
        U'' = - \frac{cW + U - U^2}{d_1 + U} = \psi_H^0(U, W; 0),
    \end{align}
    so we can still express \eqref{e: tw sys} at $\delta = 0$ as a semilinear system. Expressing the equation in quasilinear form is often more convenient for computations, but we keep in mind that we can always view it as a semilinear equation since we are restricting to $U > - d_1$. 
\end{rmk}

\subsection{Reduction of stability problem}
We emphasize that the reduction of Proposition \ref{p: existence reduction} only applies to the existence problem for traveling waves, not to the full time-dependent problem. In particular, spectral stability of fronts in the full problem \eqref{eq:main} is not determined by the linearization of the scalar equation \eqref{e: reduced tw}. 

Instead, to study spectral stability of fronts, we formulate the eigenvalue problem for \eqref{eq:main} as a first order system in $x$, and couple it to the existence problem \eqref{e: tw sys} to obtain an 8 dimensional, autonomous system depending on the eigenvalue parameter $\lambda$, and then apply a reduction similar to that of Proposition \ref{p: existence reduction} to this extended system. 

The eigenvalue problem for a traveling wave solution $(U, V)$ to \eqref{eq:main} has the form
\begin{align}
    \begin{pmatrix}
        \lambda \tilde{u} \\
        0 
    \end{pmatrix}
    = \begin{pmatrix}
        d_1 \partial_{xx} + c \partial_x & 0 \\
        1 & \delta^2 \partial_{xx} - 1
    \end{pmatrix} 
    \begin{pmatrix}
    \tilde{u} \\ \tilde{v}
    \end{pmatrix}
    + \tilde{\Theta} (U, U_x, V_x, V_{xx}) \begin{pmatrix}
    \tilde{u} \\ \tilde{v} 
        \end{pmatrix} 
        =: L(U,V) \begin{pmatrix}
        \tilde{u} \\ \tilde{v} 
        \end{pmatrix} \label{e: spectral problem}
    \end{align} 
    where
    \begin{align}
        \tilde{\Theta}(U, U_x, V_x, V_{xx}) = \begin{pmatrix}
            V_{xx} + V_x \partial_x + 1 - 2 U & U \partial_{xx} + U_x \partial_x \\
            0 & 0 
        \end{pmatrix}. 
    \end{align}
Let 
\begin{align}
    K = \begin{pmatrix}
        1 & 0 \\
        0 & 0
    \end{pmatrix}.
\end{align}
We say $\lambda$ is in the point spectrum of the generalized eigenvalue problem of  $L(U,V) - \lambda K$ if the operator is Fredholm with index 0, but not invertible. If $L(U,V) - \lambda K$ is either not Fredholm, or Fredholm with nonzero index, we say $\lambda$ is in the essential spectrum of $L(U,V)$. 
    
If $(U,V)$ is a traveling front with constant limits at $\pm \infty$, then the essential spectrum may be computed from the limiting linearized equations at $\pm \infty$; see Section \ref{s: essential spectrum}. To study point spectrum, we rewrite \eqref{e: spectral problem} as a first order system, with coordinates $(\tilde{u}, \tilde{w}, \tilde{h}, \tilde{z}) = (\tilde{u}, \tilde{u}_x,  \frac{\tilde{v}-\tilde{u}}{\delta^2}, \delta \tilde{h}_x)$, finding
\begin{align*}
    \tilde{u}' &= \tilde{w} \\
    \tilde{w}' &= -\frac{1}{d_1} D \Gamma(U, W, H, Z) (\tilde{u}, \tilde{w}, \tilde{h}, \tilde{z})^T + \frac{1}{d_1} \lambda \tilde{u} \\
    \delta \tilde{h}' &= \tilde{z}\\
    \delta \tilde{z}' &= \tilde{h} + \frac{1}{d_1} D \Gamma(U, W, H, Z) (\tilde{u}, \tilde{w}, \tilde{h}, \tilde{z})^T - \frac{1}{d_1} \lambda \tilde{u}. 
\end{align*}
To make the equation autonomous, we couple it to the existence problem \eqref{e: tw sys}, studying the 8-dimensional system,
\begin{align}
     U' &= W \nonumber ,  &\tilde{u}' &= \tilde{w} \\
    W' &= - \frac{1}{d_1} \Gamma(U, W, H, Z), &\tilde{w}' &= -\frac{1}{d_1} D \Gamma(U, W, H, Z) (\tilde{u}, \tilde{w}, \tilde{h}, \tilde{z})^T + \frac{1}{d_1} \lambda \tilde{u},  \nonumber \\
    \delta H' &= Z \nonumber, &\delta \tilde{h}' &= \tilde{z}\\
    \delta Z' &= H + \frac{1}{d_1} \Gamma(U, W, H, Z), &\delta \tilde{z}' &= \tilde{h} + \frac{1}{d_1} D \Gamma(U, W, H, Z) (\tilde{u}, \tilde{w}, \tilde{h}, \tilde{z})^T - \frac{1}{d_1} \lambda \tilde{u}. \label{e: eigenvalue sys}
\end{align}

\begin{proposition}[Reduction for eigenvalue problem]\label{p: eigenvalue reduction}
     Fix  $0 < \underline{d} < \overline{d}$, $\Lambda > 0$, and $M > 1$.  
     Then there is  $\overline{\delta} > 0$, so that the following holds. 
    All trajectories of \eqref{e: eigenvalue sys} with $|\delta| < \overline{\delta}$, $\underline{d} < d_1 < \overline{d}$, $|\lambda| < \Lambda$ satisfying
    \begin{align}
        -\frac{d_1}{2}<U<M, \quad  |W| + |H| + |Z| + \leq M,
    \end{align}
    lie in a slow manifold, which is normally hyperbolic and invariant under the flow of \eqref{e: eigenvalue sys}, and may be written as a graph
    \begin{align}
        Z &= \psi_Z (U,W; \delta), \nonumber \\
        H &= \psi_H (U, W; \delta), \nonumber \\
        \tilde{z} &= \tilde{\psi}_z (U, W, \lambda; \delta) (\tilde{u}, \tilde{w})^T \nonumber \\
        \tilde{h} &= \tilde{\psi}_h (U, W, \lambda; \delta) (\tilde{u}, \tilde{w})^T, \label{e: eigenvalue slow manifold graph}
    \end{align}
    where $\psi_{Z/H}$ are given in Proposition \ref{p: existence reduction}, and $\tilde{\psi}_{z/h} (U,W, \lambda; \delta) : \C^2 \to \C$ are linear operators which are $C^k$ in all arguments, with expansions
    \begin{align}
        \tilde{\psi}_h(U,W, \lambda; \delta) &= \tilde{\psi}_h^0(U, W, \lambda) + \mathrm{O}(\delta), \\
        \tilde{\psi}_z(U,W, \lambda; \delta) &= \mathrm{O}(\delta). 
    \end{align}

    As a consequence, in the bounded parameter regime considered here, we have a bounded solution to the eigenvalue problem \eqref{e: eigenvalue sys} if and only if we have a bounded solution to the reduced problem 
    \begin{align}
        \tilde{u}' &= \tilde{w}, \nonumber\\
        \tilde{w}' &= - \frac{1}{d_1} D \Gamma(U,W, \psi_H (U, W; \delta), \psi_Z(U, W; \delta)) \cdot (\tilde{u}, \tilde{w}, \tilde{h}, \tilde{z})^T + \frac{1}{d_1} \lambda \tilde{u}, \label{e: eigenvalue reduced sys}
    \end{align}
    where $\tilde{h}, \tilde{z}$ are given by \eqref{e: eigenvalue slow manifold graph}. 
\end{proposition}
\begin{proof}
    The result is a straightforward application of Fenichel's slow manifold reduction \cite{fenichel79} with two minor adaptations. First, Fenichel's result perturbs compact manifolds of equilbria, while we allow for unbounded, non-compact manifolds in the linear components. Second, we claim that the slow manifold preserves the structure of the original system, namely the skew-product structure where the nonlinear problem does not depend on components of the linear subsystem, and the fact that the linear subsystem is linear. Inspecting the construction of slow manifolds via graph transform, one quickly notices that noncompactness in the linear subsystem poses no difficulty since hyperbolicity is naturally uniform in the linear equation. Secondly, one finds that the slow manifold is linear in the linear subsystem simply by initializing the graph transform with linear subspaces in this component, a property that is preserved by the flow and in the limit. Similarly, the fact that the manifold associated with the nonlinear subsystem does not depend on the linear components is preserved by the flow and therefore for graphs transported by the flow, and also in  the limit of infinite time. 
\end{proof}
\begin{lemma}\label{l: tilde psi expansion}
    Along trajectories of the reduced problem \eqref{e: eigenvalue reduced sys} at $\delta = 0$, we have
 \begin{align}
   \tilde{\psi}_h^0(U,W, \lambda)\cdot (\tilde{u},\tilde{w})^T =\tilde{u}''. 
    \end{align}
\end{lemma}
\begin{proof}
    The proof is similar to the proof of Lemma \ref{l: psi expansions} and omitted here.
\end{proof}

By Lemma \ref{l: tilde psi expansion}, at $\delta = 0$ the reduced eigenvalue problem \eqref{e: eigenvalue reduced sys} may be written in the form
\begin{align}
    d_1 \tilde{u}_{xx} + U \tilde{u}_{xx} + (c + 2 U_x) \tilde{u}_x + (U_{xx} + 1 - 2U) \tilde{u} = \lambda \tilde{u}, \label{e: eigenvalue reduced scalar delta 0}
\end{align}
which is precisely the eigenvalue problem for the artificial time dependent reduction
\begin{align}
    u_t = d_1 u_{xx} + c u_x + f(u, u_x; 0, d_1). 
\end{align}

\section{Preliminaries: spreading speed and the $\delta = 0$ limit} \label{s: preliminaries}
Before studying persistence of selected fronts to $\delta > 0$, we compute the linear spreading speed associated to \eqref{eq:main} and some useful properties of the fronts in the $\delta = 0$ limit.

\subsection{Linear spreading speed}

First, we compute the linear spreading speed, at which the pulled fronts propagate. Linearizing \eqref{e: tw} about the unstable equilibrium $u = v = 0$ and making the Fourier Laplace ansatz $u,v \sim e^{\nu x + \lambda t}$, we find the \emph{dispersion relation}
\begin{align}
    d_c (\lambda, \nu; \delta, d_1) = \det \begin{pmatrix}
    d_1 \nu^2 + c \nu + 1 - \lambda & 0 \\
    1 & \delta^2 \nu^2 - 1. 
    \end{pmatrix}
    = (d_1 \nu^2 + c \nu + 1 - \lambda) (\delta^2 \nu^2 - 1). 
\end{align}
The linear spreading speed, which characterizes marginal pointwise stability in the linearization about the unstable equilibrium, is in turn characterized by \emph{pinched double roots} of the dispersion relation; see \cite{HolzerScheelPointwiseGrowth} for a thorough treatment of pointwise stability criteria, or \cite[Section 1.2]{CAMS} for a brief overview of linear spreading speeds and pinched double roots. 
\begin{lemma}[Linear spreading speed via simple pinched double root]\label{l: linear spreading speed}
    Fix $d_1 > 0$ and $\delta \geq 0$. For $c = \clin := 2 \sqrt{d_1}$, the dispersion relation $d_c(\lambda, \nu; \delta, d_1)$ has a simple pinched double root at $(\lambda_\mathrm{lin}, \nulin) := (0, -\frac{1}{\sqrt{d_1}})$. That is, for $\lambda, \nu$ small the dispersion relation admits the expansions
    \begin{align}
        d_{\clin} (\lambda, \nulin + \nu) =  d_{10} \lambda - d_{02} \nu^2 + \mathrm{O}(\nu^3, \lambda \nu) \label{e: dispersion expansion},
    \end{align}
    where $d_{10} = 1 - \delta^2 \nulin^2, d_{02} = d_1 (1-\delta^2\nulin^2)$. 
\end{lemma}
\begin{proof}
    Focusing on the first factor in the dispersion relation $\tilde{d}_c(\lambda, \nu; d_1) = d_1 \nu^2 + c \nu + 1$, and solving $\tilde{d}_c (\lambda, \nu; d_1) = \partial_\nu \tilde{d}_c(\lambda, \nu; d_1) = 0$, we readily find a double root at $(0, \nulin)$. Expanding the full dispersion relation near this double root readily gives \eqref{e: dispersion expansion}. 
\end{proof}

Actually, the expansion \eqref{e: dispersion expansion} only implies that $(\lambda, \nu) = (0, \nulin)$ is a simple double root, that is, simple in $\lambda$ and double in $\nu$. The pinching condition is then implied by the expansion \eqref{e: dispersion expansion} together with Lemma \ref{l: leading edge stability}, below. 

In fact, the dispersion relation at the linear spreading speed has the explicit form
\begin{align}
    d_{\clin} (\lambda, \nulin + \nu) = (-\lambda + d_1 \nu^2)(\delta^2 (\nu+\nulin)^2 -1). \label{e: dispersion formula}
\end{align}

\subsection{Essential spectrum and exponential weights}\label{s: essential spectrum}
\paragraph{Essential spectrum in the leading edge.} The spectrum $\Sigma_+$ of the linearization about $u = v = 0$, for instance in $L^2(\R)$, is readily determined by the dispersion relation via the Fourier transform, as
\begin{align}
    \Sigma_+ = \{ \lambda \in \C : d_{\clin} (\lambda, ik; \delta, d_1) = 0 \text{ for some } k \in \R \}. 
\end{align}
Substituting for instance $k = 0$, one readily sees that the spectrum of $u = v = 0$ is unstable. We can recover marginal spectral stability at the linear spreading speed, however, using exponential weights. Indeed, let 
\begin{align}
    \mathcal{A}_+ = \begin{pmatrix}
        d_1 \partial_x^2 + \clin \partial_x + 1 & 0 \\
        1 & \delta^2 \partial_x^2 - 1
    \end{pmatrix}
\end{align}
denote the linearization about $u = v = 0$, and let $\etalin = -\nulin = \frac{1}{\sqrt{d_1}}$. Then, defining the conjugate operator $\mcl_+u = e^{\etalin \cdot} [\mathcal{A}_+ [e^{-\etalin \cdot} u(\cdot)]]$, equivalent to considering $\mathcal{A}_+$ acting on a weighted space with weight $e^{\etalin x}$, we find 
\begin{align}
    \mcl_+ = \begin{pmatrix}
        d_1 \partial_x^2 & 0 \\
        1 & \delta^2 (\partial_x^2 - 2 \etalin \partial_x + \etalin^2) - 1
    \end{pmatrix},
\end{align}
with associated spectrum
\begin{align}
    \Sigma^{\etalin}_+ = \{ \lambda \in \C : d_{\clin} (\lambda, ik - \etalin; \delta, d_1) = 0 \text{ for some } k \in \R\}. 
\end{align}
Inspecting the modified dispersion relation, we readily find marginal spectral stability in the leading edge, summarized in the following lemma. 
\begin{lemma}[Marginal stability in the leading edge]\label{l: leading edge stability}
    We have 
    \begin{align}
        \Sigma_+^{\etalin} = \{ -k^2 \in \C : k \in \R \}. 
    \end{align}
\end{lemma}

\paragraph{Essential spectrum in the wake.}Linearization \eqref{eq:main} about $u \equiv v \equiv 1$, we find essential spectrum
\begin{align}
    \Sigma_- = \{ \lambda \in \C : d^-(\lambda, ik; \delta, d_1) = 0 \text{ for some } k \in \R \},
\end{align}
where
\begin{align}
    d^-(\lambda, \nu) = \det \begin{pmatrix}
        d_1 \nu^2 + \clin \nu - 1 - \lambda & 0 \\
        1 & \delta^2 \nu^2 -1
    \end{pmatrix} = (d_1 \nu^2 + \clin \nu - 1 - \lambda)(\delta^2 \nu^2 -1).
\end{align}
We readily find that the spectrum associated to this state is stable.

\begin{lemma}[Stability in the wake] \label{l: wake stability}
    The spectrum $\Sigma_-$ of the linearization about $u \equiv v \equiv 1$ is strictly contained in the open left half plane. 
\end{lemma}

\paragraph{Two-sided exponential weights.} Since we need an exponential weight to stabilize the essential spectrum in the leading edge but not in the wake, we define modified weights with separate growth rates on $x > 0$ and $x < 0$. Given $\eta_\pm \in \R$, we define a smooth positive weight function $\omega_{\eta_-, \eta_+} (x)$ satisfying
\begin{align}
    \omega_{\eta_-, \eta_+} (x) = \begin{cases}
        e^{\eta_+ x}, & x \geq 1, \\
        e^{\eta_- x}, & x \leq -1. 
    \end{cases}
\end{align}
Given non-negative integers $k$ and $m$, we then define the weighted Sobolev space $H^k_{\eta_-, \eta_+} (\R, \C^m)$ through the norm
\begin{align}
    \| g \|_{H^k_{\eta_-, \eta_+} (\R, \C^m)} = \| \omega_{\eta_-, \eta_+} g \|_{H^k (\R, \C^m)}.
\end{align}
When $k = 0$, we denote $H^0_{\eta_-, \eta_+} (\R, \C^m) = L^2_{\eta_-, \eta_+} (\R, \C^m)$. When it is clear from context what value of $m$ we are considering, we will abbreviate $H^k_{\eta_-, \eta_+} (\R, \C^m) = H^k_{\eta_-, \eta_+}$. We will use the notation $\omega_{\eta_-, \eta_+}^{-1}$ to denote the reciprocal function $x \mapsto \frac{1}{\omega_{\eta_-, \eta_+} (x)}$, rather than the inverse function of $x \mapsto \omega_{\eta_-, \eta_+}(x)$.

\subsection{Refined properties of fronts in the porous medium limit}
In order to establish persistence of invasion fronts for $\delta \neq 0$, we will need some finer properties of the fronts in the porous medium limit $\delta = 0$ than those captured in Lemma \ref{lem:PME}. We record these properties here, and establish existence and spatial asymptotics for these fronts in Appendix \ref{a: pme limit}. The essential spectrum and associated Fredholm properties may be computed from the asymptotic dispersion relations by Palmer's theorem \cite{Palmer2} (see e.g. any of \cite{KapitulaPromislow, FiedlerScheel, SandstedeReview} for a review), and stability of point spectrum can be established with Sturm-Liouville arguments, so we record only the results here. 

\begin{lemma}[Marginal stability of pulled fronts in the porous medium limit]\label{l: pme properties pulled}
    Fix $d_1 > \frac{1}{2}$, and let $\upl^0$ be the unique (up to translation) front solution with speed $c = \clin$ to \eqref{e: pme} guaranteed by Lemma \ref{lem:PME}. Let $\mca_\mathrm{pl}^0$ denote the linearization of \eqref{e: pme} about this front solution, and define $\mcl_\mathrm{pl}^0 = \omega_{0, \etalin} \mca_\mathrm{pl}^0 \omega_{0, \etalin}^{-1}$. Then:
    \begin{itemize}
        \item The pulled front $\upl^0$ has asymptotics
        \begin{align}\label{e:pmeanonzero}
            \upl^0(x) \sim (ax+b) e^{-\etalin x}, \quad x \to \infty
        \end{align}
        with $a > 0$. Translating the front in space, we may assume $b = 1$. 
        \item The essential spectrum of $\mcl_\mathrm{pl}^0$ is marginally stable, consisting of the union of $\Sigma_+^{\etalin}$ with the subset of the complex plane which lies on and to the left of the parabola $\Sigma_-$. 
        \item $\mcl_\mathrm{pl}^0$ has no eigenvalues $\lambda$ with $\Re \lambda \geq 0$, and there is no bounded solution to $\mcl_\mathrm{pl}^0 u = 0$. 
        \item If $\eta = \etalin + \tilde{\eta}$ with $\tilde{\eta}$ small, then $\mca_\mathrm{pl}^0: H^2_{0, \eta} \to L^2_{0, \eta}$ is Fredholm with index -1 with trivial kernel and one dimensional cokernel. 
    \end{itemize}
\end{lemma}

\begin{lemma}[Marginal stability of pushed fronts in the porous medium limit]\label{l: pme properties pushed}
    Fix $d_1 < \frac{1}{2}$, and let $\ups^0$ denote the unique (up to translation) front solution to \eqref{e: pme} traveling with the pushed speed $\cpm(d_1) = \frac{1}{\sqrt{2}} + \sqrt{2} d_1$. Fix $\etalin < \eta_0 < \sqrt{\frac{1}{2 d_1^2}}$, let $\mca_\mathrm{ps}^0$ denote the linearization of \eqref{e: pme} about $\ups^0$, and define the weighted linearization $\mcl_\mathrm{ps}^0 = \omega_{0, \eta_0} \mca_\mathrm{ps}^0 \omega_{0, \eta_0}^{-1}$. Then:
    \begin{itemize}
        \item The pushed front $\ups^0$ has asymptotics
        \begin{align}
            \ups^0(x) \sim e^{-\etaps x}, \quad x \to \infty,
        \end{align}
        where $\etaps = - \frac{1}{\sqrt{2 d_1^2}}$, possibly after translating in space. 
        \item The essential spectrum of $\mcl_\mathrm{ps}^0$ is strictly contained in the open left half plane. 
        \item $\lambda = 0$ is a simple eigenvalue of $\mcl_\mathrm{ps}^0$, with eigenfunction $u = \omega_{0, \eta_0} \partial_x {\ups^0}$ arising from translation invariance.
        \item In particular, the previous two results imply that operator $\mca_\mathrm{ps}^0 : H^2_{0, \eta_0} \to L^2_{0, \eta_0}$ is Fredholm with index 0, one-dimensional kernel, and one-dimensional cokernel. 
        \item Apart from $\lambda = 0$, $\mcl_\mathrm{ps}^0$ has no other eigenvalues with $\Re \lambda \geq 0$. 
    \end{itemize}
\end{lemma}

\begin{lemma}[Pushed-pulled transition in the porous medium limit]\label{l: transition pme limit}
    Let $\utr$ denote the unique (up to translation) solution to \eqref{e: pme} with $d_1 = \frac{1}{2}$, traveling with the speed $\cpm(\frac{1}{2}) = \sqrt{2}$. Let $\Atr^0$ denote the linearization of \eqref{e: pme} about $\utr^0$, and define the weighted linearization $\Ltr^0 = \omega_{0, \etalin} \Atr^0 \omega_{0, \etalin}^{-1}$. Then:
    \begin{itemize}
        \item The essential spectrum of $\Ltr^0$ is marginally stable, consisting of the union of $\Sigma_+^{\etalin}$ with the subset of the complex plane which lies on and to the left of the parabola $\Sigma_-$. 
        \item $\Ltr^0$ has no eigenvalues with $\Re \lambda \geq 0$, and the function $u = \omega_{0, \etalin} \partial_x \utr^0$ is the unique bounded solution to $\Ltr^0 u = 0$, up to a scalar multiple. 
        \item For $\eta > \etalin$, the operator $\Atr^0 : H^2_{0, \eta} \to L^2_{0, \eta}$ is Fredholm with index -1, trivial kernel, and one-dimensional cokernel. For $\eta_2 < \etalin$, the operator $\Atr^0 : H^2_{0, \eta_2} \to L^2_{0, \eta_2}$ is Fredholm with index 1, trivial cokernel, and one-dimensional  kernel. 
    \end{itemize}
\end{lemma}

To analyze the perturbation to $\delta \neq 0$, we will need to project various quantities onto the cokernel of the linearization about a given front, which we describe in the following lemma. 
\begin{lemma}\label{l: cokernel} 
    For $j \in \{ \mathrm{pl}, \mathrm{ps}, \mathrm{tr} \}$, let $u_j$ denote the associated front described in Lemmas \ref{l: pme properties pulled} through \ref{l: transition pme limit}, with associated speed $c_j$, and let $\mca_j^0$ denote the corresponding linearization, considered on one of the weighted spaces described in Lemmas \ref{l: pme properties pulled} through \ref{l: transition pme limit} such that $\mca_j^0$ has a one-dimesional cokernel. This cokernel is spanned by
    \begin{align}
        \phi_j (x) := \frac{\rho_j(x)^2 \partial_x u_j (x)}{d_1 + u_j(x)}, \label{e: cokernel def}
    \end{align}
    where 
    \begin{align}
        \rho_j(x) = \exp (-m_j(x)),\quad m_j(x)=-\int_{x_0}^x  \frac{2u_j'(x) +c_j}{ 2(d_1+u_j(x))} dx, \label{e: m integral}
    \end{align}
    where $x_0$ is the unique point such that $u_j(x_0) = \frac{1}{2}$. 
\end{lemma}
\begin{proof}
    This follows by first dividing the operator $A_j^0$ by $d_1 + u_j(x)$, which is strictly positive, to put the second-order term in divergence form, and then verifying that the resulting operator becomes self-adjoint after conjugation with the non-nonegative weight $\rho_j$, and recalling that $A_j^0 (\partial_x u_j^0) = 0$ by translation invariance. 
\end{proof}



\section{Persistence of selected fronts}\label{s: existence}

\subsection{Persistence of pulled fronts}
With the reduction of Section \ref{s: regularization} in hand and the detailed description of the $\delta = 0$ limit from Section \ref{s: preliminaries}, we now establish persistence of pushed and pulled fronts for $\delta > 0$. 

\begin{proposition}[Persistence of pulled fronts away from the pushed-pulled transition]\label{p: pulled persistence}
    Fix $d_1 > \frac{1}{2}$. Then there exist $\delta_1 = \delta_1 (d_1) > 0$, depending continuously on $d_1$, and an $\eta > 0$ such that if $|\delta| < \delta_1$, the reduced equation \eqref{e: reduced tw} with $c = \clin$ has a pulled front solution $\upl(x; \delta)$ such that $\upl(x;\delta)$ converges to $\upl^0$ in $C^k_\mathrm{loc}$ as $\delta \to 0$ and has generic asymptotics
        \begin{align}
            \upl(x; \delta) &\sim (a(\delta) x + b(\delta)) e^{-\etalin x}, \quad x \to \infty, \\
            \upl(x; \delta) &\sim 1 + \mathrm{O}(e^{\eta x}), \quad x \to -\infty. 
        \end{align}
        where $a(\delta), b(\delta)$ are $C^1$ in $\delta$, and $a(0) > 0$. 
\end{proposition}
\begin{proof}
    Consider the reduced traveling wave equation and the associated (artifical) parabolic equation
    \begin{align}
        u_t = d_1 u_{xx} + \clin u_x + f(u, u_x; \delta, d_1). \label{e: reduced time dependent} 
    \end{align}
    By Lemma \ref{l: pme properties pulled}, the front $\upl^0$ is a generic pulled front in \eqref{e: reduced time dependent} in the sense of \cite{CAMS}: that is, it travels with the linear spreading speed, has weak exponential decay in the leading edge, has marginally stable essential spectrum in the weighted space with weight $e^{\etalin x}$, and has stable point spectrum. Since \eqref{e: reduced time dependent} is a scalar, semilinear parabolic equation which depends continuously on $\delta$, this pulled front persists as a marginally stable pulled front in \eqref{e: reduced time dependent} for $\delta \neq 0$ by \cite[Theorem 2]{CAMS}. Note that this marginal stability holds only when we view the pulled front as a solution of the artificial time dependent equation \eqref{e: reduced time dependent}: we have not yet established marginal spectral stability for the linearization about the associated front in the full problem \eqref{eq:main}. Although Theorem 2 of \cite{CAMS} is stated only for $f = f(u; \delta)$, the proof carries over to the general semilinear case with straightforward modifications to the notation. 
\end{proof}

\subsection{Persistence of pushed fronts}\label{s: pushed persistence}

\begin{proposition}[Persistence of pushed fronts away from the pushed-pulled transition]\label{p: pushed persistence}
    Fix $d_1 < \frac{1}{2}$. Then there exist $\delta_1 = \delta_1(d_1) > 0$, depending continuously on $d_1$, and $\eta > 0$ and a speed $\cps(\delta)$, $C^1$ in $\delta$ for $|\delta| < \delta_0$ such that for $|\delta| < \delta_0$, the reduced equation \eqref{e: reduced tw} has a pushed front solution $\ups(x; \delta)$ such that $\ups(x; \delta)$ converges to $\ups^0$ in $C^k_\mathrm{loc}$ as $\delta \to 0$, and has generic asymptotics
    \begin{align}
        \ups(x; \delta) &\sim e^{-\etaps x}, \quad x \to \infty, \\
        \ups(x; \delta) &\sim 1 + \mathrm{O}(e^{\eta x}), \quad x \to -\infty. 
    \end{align}
\end{proposition}
\begin{proof}
    This is a simple bifurcation theory argument, but we repeat it in order to compute asymptotics for the pushed speed. We fix $\etalin < \eta_0 < \etaps$, and look for front solutions to \eqref{e: reduced tw} in the form of the ansatz
    \begin{align}
        \ups(x) = \chi_-(x) + v(x), 
    \end{align}
    where $\chi_-(x)$ is a smooth positive cutoff function satisfying
    \begin{align} \label{e : chiminus}
        \chi_- (x) = \begin{cases}
            1, & x \leq -3, \\
            0, & x \geq -2,
        \end{cases}
    \end{align}
    and we require $v \in H^2_{0, \eta_0} (\R, \R)$. Inserting this ansatz into \eqref{e: reduced tw} leads to an equation $F(v, c; \delta) = 0$, where 
    \begin{align}
        F_{\mathrm{ps}} : \mathcal{U} \subset H^2_{0, \eta_0} \times \R \times (-\overline{\delta}, \overline{\delta}) \to L^2_{0, \eta_0}
    \end{align}
    is $C^1$ in all arguments. Note that $F_{\mathrm{ps}}(v_0; \cpm(d_1), 0) = 0$, with $v_0 = \ups^0 - \chi_-$. Here, $\mathcal{U}$ is a sufficiently small neighborhood of $v_0$ such that 
    \begin{align*}
        - \frac{d_1}{2} < \chi_- + v_0 < M, \quad |\chi_-' + v_0'| \leq M,
    \end{align*}
    and hence $\ups = \chi_- + v_0$ remains in the region where the reduction of Proposition \ref{p: existence reduction} is valid.
    
    Linearizing about $v_0$, we find $D_v F_{\mathrm{ps}}(v_0; \cpm(d_1), 0) = \mca_\mathrm{ps}^0 : H^2_{0, \eta_0} \subset L^2_{0, \eta_0} \to L^2_{0, \eta_0}$ which is Fredholm with index 0 by Lemma \ref{l: pme properties pushed}. The linearization is not invertible in this space, however, since the translation invariance gives rise to a kernel, $\mca_\mathrm{ps}^0 (\partial_x \ups^0) = 0$, and since $\eta_0 < \etaps$, the latter of which is the decay rate of $\ups^0$, we have $\partial_x \ups^0 \in H^2_{0, \eta_0}$. 

    To recover invertibility, we add an additional equation which fixes the spatial translation of solutions, and add the speed $c$ as an additional variable to compensate. That is, we define
    \begin{align}
        G_{\mathrm{ps}}(v, c; \delta) = \begin{pmatrix}
            F_{\mathrm{ps}}(v, c; \delta) \\
            \langle \chi_-+v, \partial_x \ups^0 \rangle - \langle \ups^0, \partial_x \ups^0 \rangle 
        \end{pmatrix}.
    \end{align}
    We still have a solution $G_{\mathrm{ps}}(v_0, \cpm(d_1); 0) = 0$ with $\delta = 0$. Linearizing in the joint variable $(v, c)$ at this solution, we find
    \begin{align}
        D_{(v,c)} G_{\mathrm{ps}}(v_0, \cpm(d_1); 0) = \begin{pmatrix} 
        \mca_\mathrm{ps}^0 & \partial_x \ups^0 \\
        \langle \cdot, \partial_x \ups^0 \rangle & 0
        \end{pmatrix}. 
    \end{align}
    By the Fredholm bordering lemma, $D_{(v,c)} G_{\mathrm{ps}}(v_0, \cpm(d_1); 0)$ is Fredholm with index 0, and so is invertible if and only if it has trivial kernel. A pair $(\tilde{v}, \tilde{c})$ belongs to the kernel of this linearization if and only if 
    \begin{align}
        \mca_\mathrm{ps}^0 \tilde{v} + \tilde{c} \partial_x \ups^0 &= 0, \\
        \langle \tilde{v}, \partial_x \ups^0 \rangle &= 0. 
    \end{align}
    By Lemma \ref{l: pushed continuation linear independence}, below, $\partial_x \ups^0$ is not in the range of $\mca_\mathrm{ps}^0$, and hence the first equation can only be satisfied if $\tilde{c} = 0$ and $\tilde{v} = \alpha \partial_x \ups^0$ for some $\alpha \in \R$. But then the second condition becomes $\alpha \| \partial_x \ups^0 \|_{L^2}^2 = 0$, which can only hold if $\alpha = 0$, so the kernel must be trivial. The desired result then follows by applying the implicit function theorem. 
\end{proof}

\begin{lemma}\label{l: pushed continuation linear independence}
    Fix $d_1 < \frac{1}{2}$ and $\etalin < \eta_0 < \etaps$. The function $\partial_x \ups^0$ is not in the range of $\mca_\mathrm{ps}^0 : H^2_{0, \eta_0} \to L^2_{0, \eta_0}$. 
\end{lemma}
\begin{proof}
    Recall from Lemmas \ref{l: pme properties pushed} and \ref{l: cokernel} that $\mca_\mathrm{ps}^0$ on $H^2_{0, \eta_0}$ is Fredholm with trivial kernel and one-dimensional cokernel spanned by the function $\phi_\mathrm{ps}$ given in \eqref{e: cokernel def}. 

    The desired condition is then equivalent to $\langle \partial_x \ups^0, \phi_\mathrm{ps} \rangle_{L^2} \neq 0$. We then conclude
    \begin{align}
        \langle \partial_x \ups^0, \phi_\mathrm{ps} \rangle = \int_\R \frac{\rho(x)^2 | \partial_x \ups^0 (x)|^2}{d_1 + \ups^0(x)} \, dx > 0, 
    \end{align}
    as desired. 
\end{proof}

Having established persistence of the pushed fronts, we now compute the leading order expansion of the pushed front speed. From Proposition \ref{p: pushed persistence}, we find in particular a solution $(\vps(\delta), \cps(\delta))$ to $F(\vps(\delta), \cps(\delta); \delta) = 0$, which we may write as $\vps(\delta) = v_0 + \tilde{v}, \cps(\delta) = \cpm(d_1) + \tilde{c}$, where $\tilde{v}$ and $\tilde{c}$ are both at least order $\delta$. Expanding the equation $F(v_0 + \tilde{v}, \cpm(d_1) + \tilde{c}; \delta) = 0$, and in particular using Lemma \ref{l: psi expansions} to express expansions of $f(u, u_x; \delta)$ in terms of higher derivatives of $\ups^0$, we find
\begin{align}
    0 &= \mca_\mathrm{ps}^0 \tilde{v} + \tilde{c} \partial_x \ups^0 + \delta^2 (\partial_x \ups^0 \partial_x^3 \ups^0 + \ups^0 \partial_x^4 \ups^0) + \mathrm{O}(|\tilde{v}|^2, \delta^2 |\tilde{v}|,  \tilde{c}^2, \delta^4) \\
    0 &= \mca_\mathrm{ps}^0 \tilde{v} + \tilde{c} \partial_x \ups^0 + \delta^2 \partial_x (\ups^0 \partial_x^3 \ups^0) + \mathrm{O}(|\tilde{v}|^2, \delta^2 |\tilde{v}|, \tilde{c}^2, \delta^4). 
\end{align}
From this, we may infer that $\tilde{v}, \tilde{c}$ are both actually order $\delta^2$. Projecting onto the cokernel eliminates the first term, leaving only the scalar equation
\begin{align}
    \tilde{c} = -\frac{\langle \partial_x (\ups^0 \partial_x^3 \ups^0), \phi_\mathrm{ps} \rangle}{\langle \partial_x \ups^0, \phi_\mathrm{ps} \rangle} \delta^2 + \mathrm{O}(\delta^4),
\end{align}
and hence
\begin{align}
    \cps(\delta) = \cpm(d_1) -\frac{\langle \partial_x (\ups^0 \partial_x^3 \ups^0), \phi_\mathrm{ps} \rangle}{\langle \partial_x \ups^0, \phi_\mathrm{ps} \rangle} \delta^2 + \mathrm{O}(\delta^4). \label{e:pexp}
\end{align}
We now explicitly evaluate the integrals in the scalar products in \eqref{e:pexp}. We will do this by re-expressing integrals over $x$ as integrals over $u$, which turn out to be integrals of rational functions which may be computed explicitly. The key observation is that in the pushed front regime $d_1 \leq \frac{1}{2}$, one may explicitly solve for the inverse of $x \mapsto \ups^0(x)$, finding $\psi(u(x)) = x$ with
\begin{align*}
    \psi(u)=\sqrt{2}((1 + d_1) \log(1 - u) - d_1 \log(u));
\end{align*}
see \eqref{e:fexpl}. Note that by choosing this expression for the inverse, we are fixing a particular spatial translate of $\ups^0$ such that $\ups^0 (x_0) = \frac{1}{2}$, where
\begin{align*}
    x_0 = \psi \left( \frac{1}{2} \right) = - \sqrt{2} \log 2. 
\end{align*}
All $x$ derivatives of $\ups^0$ in \eqref{e:pexp} may be expressed as functions of $u$, e.g. 
\begin{align*}
    \partial_x \ups^0(x) = \frac{1}{\psi'(\ups^0(x))}. 
\end{align*}
Evaluating the resulting integral of rational functions of $u$, we find
\begin{align}
    \langle \partial_x \ups^0, \phi_\mathrm{ps} \rangle = \frac{4 \sqrt{2} \pi  d_1 (d_1+1) \csc (2 \pi 
   d_1)}{6 d_1+3}
\end{align}
Similarly, expressing higher derivatives of $\ups^0$ as rational functions of $u$ and evaluating the resulting integral, we find
\begin{align*}
    \langle \partial_x (\ups^0 \partial_x^3 \ups^0), \phi_\mathrm{ps} \rangle = \frac{\pi  d_1 \left(-18 (d_1+1)^{2 d_1+3} d_1^{-2
   d_1}+(2 d_1 (71 d_1+134)+149) d_1+23\right) \csc (2
   \pi  d_1)}{3 (2 d_1+1)^2}
\end{align*}
Dividing these two expressions, we find
\begin{align}
        \cps(\delta) &= \cpm(d_1) +c_\mathrm{ps,2}\delta^2 + \mathrm{O}(\delta^4), \nonumber \\ c_\mathrm{ps,2}&= -\frac{(6 d_1+3) \left(-18 (d_1+1)^{2 d_1+3} d_1^{-2
   d_1}+(2 d_1 (71 d_1+134)+149) d_1+23\right)}{12
   \sqrt{2} (d_1+1) (2 d_1+1)^2},\label{e:pexp2}
\end{align}
which concludes the proof of the statement on pushed fronts in Theorem \ref{t: main}, up to the marginal stability, which is established in Section \ref{s: stability}.

\subsection{Persistence of the pushed-pulled transition}

By Lemma \ref{l: transition pme limit}, the limiting porous medium equation undergoes a pushed-pulled transition at $d_1 = \frac{1}{2}$. To continue this transition point in $\delta$, we look for solution to the reduced traveling wave equation \eqref{e: reduced tw} in the form of the ansatz
\begin{align}
    u (x) = \chi_-(x) + w(x) + \chi_+(x) (ax + b) e^{-\etalin x}, \label{e: pulled ansatz}
\end{align}
where $a, b \in \R$, $\chi_+(x)=\chi_-(-x)$, and $w \in H^2_{0, \eta} (\R, \R)$ for $\eta = \etalin+\tilde{\eta}$ with $\tilde{\eta} > 0$ small. We will use the implicit function theorem to solve for $w, a,$ and $b$ as functions of $d_1$ and $\delta$. The pushed-pulled transition is then characterized by strong leading edge decay, $a(d_1, \delta) = 0$ \cite{avery22}. 

Inserting the ansatz \eqref{e: pulled ansatz} into \eqref{e: reduced tw}, we arrive at an equation
\begin{align}
    F_\mathrm{pl}(w, b, a; d_1, \delta) = 0. 
\end{align}
The form of the ansatz, capturing explicitly leading order behavior in the wake and leading edge, guarantees that $F_\mathrm{pl}$ preserves exponential localization of $w$, as follows.
\begin{lemma}
    Fix $\tilde{\eta} > 0$ small, and set $\eta = \etalin + \tilde{\eta}$. There exists $\eps > 0$ and a neighborhood $\mathcal{U} \subset H^2_{0, \eta}$ of the function $w_0 = \utr^0 - \chi_- - \chi_+ e^{\etalin x}$ such that
    \begin{align}
        F_\mathrm{pl} : \mathcal{U} \times \R^2 \times \left( \frac{1}{2} - \eps, \frac{1}{2} + \eps \right) \times (-\eps, \eps) \to L^2_{0, \eta} 
    \end{align}
    is well-defined and $C^k$ in all arguments. 
\end{lemma}
Restricting to a neighborhood $\mathcal{U}$ of $w_0$ again ensures that the reduction of Proposition \ref{p: existence reduction} remains valid for the solutions considered here. 

When $\delta = 0$, by Lemma \ref{l: transition pme limit}, we have a solution $F_\mathrm{pl} (w_0, 1, 0; \frac{1}{2}, 0)$. The linearization in $w$ about this solution is $D_w F_\mathrm{pl} (w_0, 1, 0; \frac{1}{2}, 0) = \Atr^0$, which is Fredholm with index -1 by Lemma \ref{l: transition pme limit}. By the Fredholm bordering lemma, the joint linearization $D_{(w,a,b)} F_\mathrm{pl} (w_0, 1, 0; \frac{1}{2}, 0)$ is then Fredholm with index 1. To recover invertibility, we therefore need to add an extra condition which fixes the spatial translate of the solutions, so we define
\begin{align}
    G_\mathrm{pl} (w, b, a; d_1, \delta) = \left(\begin{array}{c}
        F_\mathrm{pl}(w, b, a; d_1, \delta) \\
        \langle w + (ax + b) \chi_+ e^{-\etalin x}, \partial_x \utr^0 \rangle - \langle \utr^0, \partial_x \utr^0 \rangle 
    \end{array}\right)
\end{align}
and solve $G_\mathrm{pl}(w,b,a, d_1, \delta) = 0$. Note that we still have $G_\mathrm{pl}(w_0, 1, 0; \frac{1}{2}, 0) = 0$. 
\begin{lemma}
    The joint linearization 
    \begin{align}
        D_{(w,a,b)} G_{\mathrm{pl}} (w_0, 1, 0; \frac{1}{2}, 0) : H^2_{0, \eta} \times \R^2  \to L^2_{0, \eta}
    \end{align}
    is invertible. 
\end{lemma}
\begin{proof}
    By the Fredholm bordering lemma, this joint linearization is Fredholm with index 0, so to prove that it is invertible we need to check that the kernel is trivial. From a short computation, we find
    \begin{align}
        D_{(w,a,b)} G (w_0, 1, 0; \frac{1}{2}, 0) = \begin{pmatrix}
            \Atr^0  & \Atr^0 (x \chi_+ e^{-\etalin x}) & \Atr^0 (\chi_+ e^{-\etalin x}) \\
            \langle \cdot, \partial_x \utr^0 \rangle  & \langle x \chi_+ e^{-\etalin x}, \partial_x \utr^0 \rangle & \langle \chi_+ e^{-\etalin x}, \partial_x \utr^0 \rangle
        \end{pmatrix},
    \end{align}
    hence an element $(\tilde{w}, \tilde{a}, \tilde{b})$ of the kernel satisfies
    \begin{align}
        \Atr^0 [\tilde{w} + (\tilde{a} x + \tilde{b}) \chi_+ e^{-\etalin x}] &= 0, \\
        \langle \tilde{w} + (\tilde{a} x + \tilde{b}) \chi_+ e^{-\etalin x}, \partial_x \utr^0 \rangle &= 0. 
    \end{align}
    It follows from Lemma \ref{l: transition pme limit} that the only solution to $\Atr^0 u = 0$ for which $e^{\etalin \cdot} u$ is at most polynomially growing in $x$ is $\partial_x \utr^0$, up to a constant multiple. Since $\partial_x \utr^0 (x) \sim -\etalin e^{-\etalin x}$ as $x \to \infty$, we must have $\tilde{a} = 0$, and $\tilde{w} + \tilde{b} \chi_+ e^{-\etalin x} = \alpha \partial_x \utr^0$ for some constant $\alpha \in \R$. The second equation then becomes
    \begin{align}
        \alpha \langle \partial_x \utr^0, \partial_x \utr^0 \rangle = 0,
    \end{align}
    which implies we must have $\alpha = 0$, so the kernel is trivial, as desired. 
\end{proof}

Using the implicit function theorem, we readily obtain the following result. 

\begin{corollary}[Persistence of pulled fronts near the pushed-pulled transition]\label{c: pulled persistence near transition}
    There exists $\eps > 0$ such that for all $d_1, \delta$ such that $|d_1 - \frac{1}{2}| < \eps, |\delta| < \eps$, the equation \eqref{e: reduced tw} admits pulled front solutions $q_\mathrm{pl}$, with the form
    \begin{align}
        q_\mathrm{pl}(\cdot; d_1, \delta) = \chi_- + w(d_1, \delta) + (a(d_1, \delta) x + b(d_1, \delta)) \chi_+ e^{-\etalin x},
    \end{align}
    where $w(d_1, \delta) \in H^2_{0, \eta}$ and $a(d_1, \delta), b(d_1, \delta) \in \R$ are $C^k$ in both parameters, with $a(\frac{1}{2}, 0) = 0$ and $b(\frac{1}{2}, 0) = 1$. 
\end{corollary}

\begin{proposition}[Persistence of pushed-pulled transition]\label{p: generic transition}
        There exists a $C^k$ function $d_1^*(\delta)$ such that $a(d_1, \delta) = 0$ in a neighborhood of $(d_1, \delta) = (\frac{1}{2}, 0)$ if and only if $d_1 = d_1^*(\delta)$. Moreover, $\partial_{d_1} a\left(\frac{1}{2}, \delta\right)$ is negative for $\delta$ small, and $d_1^*(\delta)$ has the expansion
        \begin{align}
            d_1^*(\delta) = \frac{1}{2} + d_{1,2}\delta^2 + \rmO(\delta^4),\qquad 
                    d_{1,2} = \frac{1}{16} \left(268-243 \log \left(3 \right)\right) \approx 0.0648259.
        \end{align}
        We denote the associated front solutions with $d_1 = d_1^*(\delta), c = \clin(d_1^*(\delta))$ by $\utr(\cdot; \delta)$. 
\end{proposition}
\begin{proof}
    Since $a\left(\frac{1}{2}, 0\right) = 0$, we can solve $a(d_1, \delta) = 0$ nearby with the implicit function theorem provided $\partial_{d_1} a (\frac{1}{2}, 0) \neq 0$. Expanding the equation $F_\mathrm{pl} (w(d_1, \delta), a(d_1, \delta), b(d_1, \delta); d_1, \delta) = 0$ as in \cite[Section 4.2]{avery22}, we find
    \begin{align}
        \partial_{d_1} a\left(\frac{1}{2}, 0\right) = \frac{ \langle \partial_x^2 \utr^0 + \clin'(\frac{1}{2}) \partial_x \utr^0 + \nulin'(\frac{1}{2}) \Atr^0 (x \chi_+ e^{\nulin (\frac{1}{2}) x}), \phi_\mathrm{tr} \rangle}{\langle \Atr^0 (x \chi_+ e^{-\etalin (\frac{1}{2}) x}), \phi_\mathrm{tr} \rangle}. \label{e: partial d1 a}
    \end{align}
    It follows by Lemma \ref{l: transition pme limit} and \cite[Lemma 2.3]{avery22} that the denominator is nonzero; we will compute this quantity explicitly below in order to compute the expansion of $d_1^*(\delta)$. 

    Expanding in $\delta$ as well, we find $\partial_{\delta} a\left(\frac{1}{2}, 0\right) = 0$, but
    \begin{align}
        \frac{1}{2}\partial_{\delta}^2 a\left(\frac{1}{2}, 0 \right) = \frac{\langle \partial_x (\utr^0 \partial_x^3 \utr^0), \phi_\mathrm{tr} \rangle}{\langle \Atr^0 (x \chi_+ e^{-\etalin (\frac{1}{2}) x}), \phi_\mathrm{tr} \rangle},
    \end{align}
    and so, combining with \eqref{e: partial d1 a}, we conclude
    \begin{align}
        d_1^*(\delta) = \frac{1}{2}  -\frac{\langle \partial_x (\utr^0 \partial_x^3 \utr^0), \phi_\mathrm{tr} \rangle}{\langle \partial_x^2 \utr^0 + \clin'(\frac{1}{2}) \partial_x \utr^0 + \nulin'(\frac{1}{2}) \Atr^0 (x \chi_+ e^{\nulin (\frac{1}{2}) x}), \phi_\mathrm{tr} \rangle} \delta^2 + \mathrm{O}(\delta^4). \label{e: d1 proof}
    \end{align}
    Using the explicit inverse $\psi(u)$ from \eqref{e:fexpl} and re-expressing all integrals except for the far-field contribution $\langle \Atr^0 (x \chi_+ e^{-\etalin (\frac{1}{2}) x}, \phi_\mathrm{tr} \rangle$ as integrals over rational functions of $u$ as in Section \ref{s: pushed persistence}, we find
    \begin{align}
        \langle \partial_x (\utr^0 \partial_x^3 \utr^0), \phi_\mathrm{tr} \rangle = \frac{1}{12} \left(-\frac{201}{2}-\frac{729}{8} \left(-\log \left(\frac{3}{2}\right)-\log (2)\right)\right) \approx -.0324129, \label{e: pp M1}
    \end{align}
    and
    \begin{align}
        \langle \partial_x^2 \utr^0 + \clin'(\frac{1}{2}) \partial_x \utr^0, \phi_\mathrm{tr} \rangle = 1. \label{e: pp M2}
    \end{align}

    It remains only to compute the far-field contribution $\nulin'(\frac{1}{2}) \langle \Atr^0 (x \chi_+ e^{-\etalin (\frac{1}{2}) x}, \phi_\mathrm{tr} \rangle$. We use the fact that $\omega_{0, \etalin} = e^{\etalin x}$ on the support of $\chi_+$ to write
    \begin{align}
        \langle \Atr^0 (x \chi_+ e^{-\etalin (\frac{1}{2}) x}, \phi_\mathrm{tr} \rangle = \langle \mcl_\mathrm{tr}^0 (x \chi_+), \tilde{\phi}_\mathrm{tr} \rangle, \label{e: farfield contribution}
    \end{align}
    where $\tilde{\phi}_\mathrm{tr} (x)= e^{-\etalin(\frac{1}{2}) x} \phi_\mathrm{tr}(x)$, and $\mcl_\mathrm{tr}^0 = \omega_{0, \etalin} \Atr \omega_{0, \etalin}^{-1}$ satisfies
    \begin{align*}
        \mcl_\mathrm{tr}^0 = \frac{1}{2} \partial_x^2 + (a_2 (x) \partial_x^2 + a_1 (x) \partial_x + a_0 (x)),
    \end{align*}
    with $a_j(x) \to 0$ exponentially as $x \to \infty$; this is a consequence of the fact that $\nulin(\frac{1}{2})$ is a double root of the dispersion relation. Since $\tilde{\phi}_\mathrm{tr}$ is in the cokernel of $\mcl_\mathrm{tr}^0$, it follows that the only terms which can contribute to \eqref{e: farfield contribution} are boundary terms from integration by parts, which may arise due to the fact that $x \chi_+$ is not spatially localized. Since the coefficients $a_j(x)$ are exponentially localized and $\tilde{\phi}_\mathrm{tr}(x)$ is bounded, only the leading term $\frac{1}{2} \partial_x^2$ in $\mcl_\mathrm{tr}^0$ may contribute boundary terms. Hence, we find after integrating by parts
    \begin{align}
        \langle \mcl_\mathrm{tr}^0 (x \chi_+), \tilde{\phi}_\mathrm{tr} \rangle = \frac{1}{2} \lim_{x \to \infty} \tilde{\phi}_\mathrm{tr} (x) = -\frac{1}{2 \sqrt{2}}. 
    \end{align}
    We compute this limit in Lemma \ref{l: farfield limit}, below. Using the fact that $\nulin'(\frac{1}{2}) = \sqrt{2}$, we find for the far-field contribution
    \begin{align*}
        \nulin'(\frac{1}{2}) \langle \Atr^0 (x \chi_+ e^{-\etalin (\frac{1}{2}) x}, \phi_\mathrm{tr} \rangle = - \frac{1}{2}. 
    \end{align*}
    Combining with \eqref{e: pp M1}, \eqref{e: pp M2}, and \eqref{e: d1 proof} we finally obtain
    \begin{align}
        d_1^*(\delta) = \frac{1}{2} + d_{1,2} \delta^2 + \mathrm{O}(\delta^4), 
    \end{align}
    with
    \begin{align}
       d_{1,2} = \frac{1}{16} \left(268-243 \log(3)\right) \approx 0.0648259.
    \end{align}
    Note also that $\partial_{d_1} a \left(\frac{1}{2}, 0 \right) < 0$. 
\end{proof}

\begin{lemma}\label{l: farfield limit}
    We have
    \begin{align*}
        \lim_{x \to \infty} \tilde{\phi}_\mathrm{tr} (x) = - \frac{1}{\sqrt{2}}. 
    \end{align*}
\end{lemma}
\begin{proof}
    Recall that $\psi(\utr^0 (x)) = x$, where $\psi$ is given by \eqref{e:fexpl}. We may use this to write
    \begin{align*}
        \tilde{\phi}_\mathrm{tr}(x) = \tilde{\Phi} (\utr^0 (x)), 
    \end{align*}
    where
    \begin{align*}
        \tilde{\Phi} (u) = e^{-\sqrt{2} \psi(u)} \frac{e^{- 2 \tilde{m}(u)}}{\psi'(u) ( u + \frac{1}{2} )}.
    \end{align*}
    Here, 
    \begin{align*}
        \tilde{m}(u) = -\frac{1}{2} \log (2-2 u)+\frac{1}{2} \log (2 u)+\frac{1}{2} \log
   \left(-\frac{u}{u-1}\right)-\log \left(u+\frac{1}{2}\right)
    \end{align*}
    may be found by converting the integral for $m_j (\psi(u))$ to an integral of rational functions of $u$ and evaluating. 
    In particular, we find
    \begin{align*}
        \lim_{x \to \infty} \tilde{\phi}_\mathrm{tr} (x) = \lim_{u \to 0^+} \tilde{\Phi}(u) = - \frac{1}{\sqrt{2}}. 
    \end{align*}
\end{proof}

\begin{corollary}\label{c: pulled marginal stability near transition}
    Since $\partial_{d_1} a \left( \frac{1}{2}, \delta \right) < 0$, it follows from \cite[Theorem 2]{avery22} that the pulled fronts constructed near the pushed-pulled transition in Corollary \ref{c: pulled persistence near transition} are marginally spectrally stable for $d_1 \leq d_1^*(\delta)$, and are unstable for $d_1 > d_1^*(\delta)$. 
\end{corollary}

Similarly, by following the analysis of \cite[Section 4]{avery22} as in the proof of Corollary \ref{c: pulled persistence near transition}, we may continue pushed fronts from the pushed-pulled transition.
\begin{proposition}[Persistence of pushed fronts near the pushed-pulled transition]\label{p: pushed persistence near transition}
    There exists $\eps > 0$, a speed $c = \tilde{c}_\mathrm{ps}(d_1, \delta)$, and a decay rate $\tilde{\eta}_\mathrm{ps}(d_1, \delta)$ such that for all $d_1, \delta$ with $|d_1 - \frac{1}{2}| < \eps, |\delta| < \eps$, the equation \eqref{e: reduced tw} admits pushed front solutions $q_\mathrm{ps}$ traveling with speed $\tilde{c}_\mathrm{ps}(d_1, \delta)$, with the form
    \begin{align}
        q_\mathrm{ps} (\cdot; d_1, \delta) = \chi_- + w_\mathrm{ps} (d_1, \delta) + \chi_+ e^{- \tilde{\eta}_\mathrm{ps} (d_1, \delta) x}, 
    \end{align}
    where $w_\mathrm{ps}(d_1, \delta)$, $\tilde{c}_\mathrm{ps}(d_1, \delta)$, and $\tilde{\eta}_\mathrm{ps}(d_1, \delta)$ are $C^k$ in both parameters. Moreover, these fronts are marginally stable if $d_1 \geq d_1^*(\delta)$, and strictly stable for $d_1 < d_1^*(\delta)$. 
\end{proposition}

\section{Marginal spectral stability --- proof of Theorem \ref{t: main}}\label{s: stability}
We now establish marginal spectral stability of the fronts constructed in the previous section, justifying the use of the pushed/pulled terminology and completing the proof of Theorem \ref{t: main}. Our strategy is to apply Proposition \ref{p: eigenvalue reduction} to reduce to a regularized problem on a slow manifold, and then use a far-field/core decomposition to track eigenvalues near the essential spectrum. One challenge is that Proposition \ref{p: eigenvalue reduction} only applies to $\lambda$ in a large ball with radius $\Lambda$, and the allowable range of $\delta$ then depends on $\Lambda$. We therefore must also exclude any unstable eigenvalues with $\large$ sufficiently large, uniformly in $\delta$, which we do in the following proposition. First, let $K$ be the matrix
\begin{align}
    K = \begin{pmatrix}
        1 & 0 \\
        0 & 0
    \end{pmatrix}. \label{e: K def}
\end{align}

\begin{proposition}\label{p: large lambda stability}
    Let $\mathcal{A}$ denote the linearization of \eqref{e: tw} about any one of the pulled, pushed, or transition fronts constructed in Section \ref{s: existence}. There exist constants $\Lambda_0, \delta_0 > 0$ and $\frac{\pi}{2} < \theta_0 < \pi$ such that the equation $(\mca - \lambda K) u = 0$ has no bounded solutions provided $|\delta| < \delta_0$ and $\lambda$ satisfies
    \begin{align}
        \lambda \in \Omega_0 := \{ \lambda \in \C : |\lambda| \geq \Lambda_0, |\mathrm{Arg} \, \lambda| \leq \theta_0 \}.
    \end{align}
\end{proposition}
In particular, Proposition \ref{p: large lambda stability} excludes any unstable eigenvalues with $|\lambda|$ sufficiently large, uniformly in $\delta$. Here this issue is not trivial due to the parabolic-elliptic and singularly perturbed structure of \eqref{eq:main}. We give a proof via dynamical systems techniques in Appendix \ref{s: large lambda stability}. 

\subsection{Stability of pushed fronts away from transition}
Fix $d_1 < \frac{1}{2}$, and let $\ups(\cdot; \delta)$ denote the associated pushed front solution to the reduced problem, constructed in Proposition \ref{p: pushed persistence}. The corresponding solution to the original traveling wave problem, \eqref{e: tw sys}, is given by 
\begin{align}
    U = \ups(\cdot; \delta), \quad V := \vps(\cdot; \delta) = \ups(\cdot; \delta) + \delta^2 \psi_H (\ups(\cdot; \delta), \partial_x \ups (\cdot; \delta); \delta).
\end{align}
Let $B_\mathrm{ps}(\delta)$ denote the linearization of \eqref{e: tw sys} about this solution.

\begin{proposition}[Marginal spectral stability of pushed fronts away from pushed-pulled transition]\label{p: pushed stability away from transition}
Fix $d_1 < \frac{1}{2}$ and $\etalin < \eta_0 < \etaps$. There exists $\delta_1 (d_1) > 0$ such that for all $0 <|\delta| < \delta_1(d_1)$, the equation
    \begin{align}
        (B_\mathrm{ps}(\delta) - \lambda K) \u = 0
    \end{align}
    has no solutions $\u \in H^2_{0, \eta_0} (\R, \C^2)$ if $\Re \lambda \geq 0$, except for a simple eigenvalue at $\lambda = 0$, with eigenfunction $\u = (\partial_x \ups(\delta), \partial_x \vps(\delta))$. 
\end{proposition}
\begin{proof}
    The candidate eigenfunctions considered here are in $H^2_{0, \eta_0} (\R, \C^2)$, so in particular are bounded. By Proposition \ref{p: eigenvalue reduction}, all bounded solutions to $(B_\mathrm{ps}(\delta) - \lambda) \u  = 0$ may be recovered from solutions of the reduced problem \eqref{e: eigenvalue reduced sys}. At $\delta = 0$, it follows from Lemma \ref{l: pme properties pushed} that \eqref{e: eigenvalue reduced sys} has no solutions in $H^2_{0, \eta_0}$ with $\Re \lambda \geq 0$ except for at $\lambda = 0$, using \eqref{e: eigenvalue reduced scalar delta 0} to relate solutions to those of the scalar problem. It follows from robustness of bounded invertibility that \eqref{e: eigenvalue reduced sys} has no solutions in $H^2_{0, \eta_0}$ for $\lambda$ away from the origin but bounded. Large eigenvalues have already been excluded by Proposition \ref{p: large lambda stability}, uniformly in $\delta$ small. In a neighborhood of the origin, the isolated simple eigenvalue at $\lambda = 0$ persists as an isolated, simple eigenvalue by standard arguments, and remains at the origin since we know the eigenfunction explicitly due to translation invariance.
\end{proof}

\subsection{Stability of pulled fronts away from transition}\label{s: pulled stability}
Fix $d_1 > \frac{1}{2}$, and let $\upl(\cdot; \delta)$ denote the associated pulled front solution to the reduced problem, constructed in Proposition \ref{p: pulled persistence}. The corresponding solution to the original traveling wave problem, \eqref{e: tw sys}, is given by 
\begin{align}
    U = \upl(\cdot; \delta), \quad V:= \vpl(\cdot; \delta) + \delta^2 \psi_H (\upl(\cdot; \delta), \partial_x \upl(\cdot; \delta), \delta). 
\end{align}
Let $B_\mathrm{pl}(\delta)$ denote the linearization of \eqref{e: tw sys}, with $c = \clin$, about this solution, and let $K$ be given by \eqref{e: K def}. 
\begin{proposition}[Marginal spectral stability of pulled fronts away from pushed-pulled transition]\label{p: pulled stability away from transition}
    Fix $d_1 > \frac{1}{2}$ and $\etalin < \eta_0 < \etaps$. There exists $\delta_2 (d_1) > 0$ such that for all $0 < |\delta| < \delta_2(d_1)$, the equation
    \begin{align}
        (B_\mathrm{pl}(\delta) - \lambda K) \u = 0 
    \end{align}
    has no solutions $\u \in H^2_{0, \etalin} (\R, \C^2)$ if $\Re \lambda \geq 0$. Moreover, there is no solution at $\lambda = 0$ which belongs to $L^\infty_{0, \etalin} (\R, \C^2)$. 
\end{proposition}
The proof of Proposition \ref{p: pulled stability away from transition} is more difficult than that of Proposition \ref{p: pushed stability away from transition} since the essential spectrum of $B_\mathrm{pl}(\delta)$ touches the origin, so we cannot apply a classical argument to rule out eigenvalues in a neighborhood of the origin. 

Nonetheless, we may still reduce our consideration to the reduced eigenvalue problem \eqref{e: eigenvalue reduced sys}, which we write as the scalar equation
\begin{align}
    d_1 \tilde{u}_{xx} + L_2(\upl, \partial_x \upl, \lambda; \delta) \tilde{u}_x + L_1 (\upl, \partial_x \upl, \lambda; \delta)  \tilde{u} = \lambda \tilde{u}, \label{e: pulled scalar eigenvalue reduced}
\end{align}
where
\begin{align}
    L_1 \tilde{u} + L_2 \tilde{u}_x = D \Gamma(\upl, \partial_x \upl, \psi_H (\upl, \partial_x \upl; \delta), \psi_Z (\upl, \partial_x \upl; \delta)) 
    \begin{pmatrix}
        \tilde{u} \\
        \tilde{u}_x \\
        \tilde{\psi}_h (\upl, \partial_x \upl, \lambda; \delta)  (\tilde{u}, \tilde{u}_x)^T \\
        \tilde{\psi}_z (\upl, \partial_x \upl, \lambda; \delta) (\tilde{u}, \tilde{u}_x)^T
    \end{pmatrix}.
\end{align}

\begin{lemma}\label{l: leading edge linearization} 
    We have $L_1(0, 0, \lambda; \delta) = 1$ and $L_2 (0, 0, \lambda; \delta) = \clin$. 
\end{lemma}
\begin{proof}
    Recall from Proposition \ref{p: existence reduction} that $\psi_Z (0, 0;\delta) = \psi_H(0, 0;\delta) = 0$. Hence
    \begin{align}
        L_1 (0, 0, \lambda; \delta)\tilde{u} + L_2(0, 0, \lambda; \delta) \tilde{u}_x = D\Gamma(0, 0, 0, 0) \begin{pmatrix}
            \tilde{u} \\
            \tilde{u}_x \\
        \tilde{\psi}_h (\upl, \partial_x \upl, \lambda; \delta)  (\tilde{u}, \tilde{u}_x)^T \\
        \tilde{\psi}_z (\upl, \partial_x \upl, \lambda; \delta) (\tilde{u}, \tilde{u}_x)^T
        \end{pmatrix}.
    \end{align}
    From \eqref{e: Gamma def}, with $c = \clin$, we find $D \Gamma(0, 0, 0, 0) = (1, \clin, 0, 0)$, from which the desired result follows. 
\end{proof}

\begin{corollary}\label{c: leading edge eigenfunction}
The equation 
\begin{align}
    \tilde{u}_{xx} + L_2 (0, 0, \lambda; \delta) \tilde{u}_x + L_1 (0, 0, \lambda; \delta) \tilde{u} = \lambda \tilde{u} 
\end{align}
admits a solution
\begin{align}
    e_+(x; \gamma) = e^{\nu_-(\gamma) x}, 
\end{align}
where
\begin{align}
    \nu_-(\gamma) = \nulin - \frac{\gamma}{d_1},
\end{align}
and $\gamma = \sqrt{\lambda}$, with $\Re \gamma \geq 0$. 

\end{corollary}
\begin{proof}
    This follows from the expression \eqref{e: dispersion formula} for the asymptotic dispersion relation together with Lemma \ref{l: leading edge linearization}. 
\end{proof}

We fix $\tilde{\eta} > 0$ small and look for solutions to \eqref{e: pulled scalar eigenvalue reduced} via the far-field/core ansatz
\begin{align}
    \tilde{u}(x)= w(x) + \beta \chi_+(x) e_+(x; \gamma), 
\end{align}
requiring $w \in H^2_{0, \eta} (\R, \C)$ with $\eta = \etalin + \tilde{\eta}$, so that if $|\gamma|$ is small, $w$ decays faster than $e_+(x; \gamma)$ as $x \to \infty$. Inserting this ansatz into \eqref{e: pulled scalar eigenvalue reduced} leads to the equation
\begin{align}
     0 = F_\mathrm{stab} (w, \beta; \gamma, \delta), \label{e: Fstab zero}
\end{align}
where
\begin{align}
    F_\mathrm{stab}(w, \beta; \gamma, \delta) := [d_1 \partial_x^2 + L_2 (\upl, \partial_x \upl, \gamma^2; \delta) \partial_x + L_1 (\upl, \partial_x \upl, \gamma^2; \delta) - \gamma^2] [w + \beta + \chi_+ e_+(x; \gamma)]. 
\end{align}

\begin{lemma}\label{l: pulled eigenvalue well defined}
    Fix $\tilde \eta > 0$ small, and set $\eta = \etalin + \tilde{\eta}$. There exists $\gamma_0 > 0$ so that the map
    \begin{align}
        F_\mathrm{stab} : H^2_{0, \eta} \times \C \times B(0, \gamma_0) \times (-\delta_2 (d_1), \delta_2(d_1)) \to L^2_{0, \eta}
    \end{align}
    is well defined, linear in $w$ and $\beta$, and $C^k$ in $\gamma$ and $\delta$. Moreover, $B_\mathrm{pl(\delta)} : H^2_{0, \etalin} (\R, \C^2) \to L^2_{0, \etalin} (\R, \C^2)$ has an eigenvalue $\lambda = \gamma^2$ to the right of its essential spectrum if and only if \eqref{e: Fstab zero} has a solution with $\Re \gamma = 0$. 
\end{lemma}
\begin{proof}
    That $F_\mathrm{stab}$ preserves exponential localization follows from exponential convergence of $\upl(x)$ to $0$ as $x \to \infty$ together with the fact that $e_+(x;\gamma^2)$ solves the eigenvalue problem in the leading edge by Corollary \ref{c: leading edge eigenfunction}. Regularity in $\gamma$ and $\delta$ follows from Proposition \ref{p: eigenvalue reduction}. Equivalence to the original eigenvalue problem follows as in \cite[Section 5]{PoganScheel}. 
\end{proof}

\begin{proposition}\label{p: pulled evans function}
    Let $\gamma_0$ be as in Lemma \ref{l: pulled eigenvalue well defined}. There exists a function $E: B(0,\gamma_0) \times (-\delta_0, \delta_0) \to \C$, which is $C^k$ in both arguments, such that: 
    \begin{itemize}
        \item $B_\mathrm{pl} (\delta) : H^2_{0, \etalin} (\R, \C^2) \to L^2_{0, \etalin} (\R, \C^2)$ has an eigenvalue $\lambda = \gamma^2$ to the right of its essential spectrum if and only if $E(\gamma, \delta) = 0$ with $\Re \gamma > 0$. 
        \item The equation $B_\mathrm{pl}(\delta) u = 0$ has a solution with $\omega_{0, \etalin} u$ bounded if and only if $E(0, \delta) = 0$. 
    \end{itemize}
\end{proposition}
\begin{proof}
    It follows from \eqref{e: eigenvalue reduced scalar delta 0} that $D_w F_\mathrm{stab} (0, 0; 0, 0) = \mca_\mathrm{pl}^0$, and recall from Lemma \ref{l: pme properties pulled} that $\mca_\mathrm{pl}^0 : H^2_{0, \eta} \to L^2_{0, \eta}$ is Fredholm with index -1, trivial kernel, and one-dimensional co-kernel spanned by $\phi_\mathrm{pl}$. Let $P_0$ denote the $L^2$-orthogonal projection onto the range of $\mca_\mathrm{pl}^0 : H^2_{0, \eta} \to L^2_{0, \eta}$, and decompose the equation $0 = F_\mathrm{stab}(w, \beta; \gamma, \delta)$ as
\begin{align}
    \begin{cases}
        P_0 [d_1 \partial_x^2 + L_2(\upl, \partial_x \upl, \gamma^2; \delta) \partial_x + L_1 (\upl, \partial_x \upl, \gamma^2; \delta) - \gamma^2 ] [w + \beta \chi_+ e_+(\cdot; \gamma)] &= 0\\
        \langle [d_1 \partial_x^2 + L_2(\upl, \partial_x \upl, \gamma^2; \delta) \partial_x + L_1 (\upl, \partial_x \upl, \gamma^2; \delta) - \gamma^2 ] [w + \beta \chi_+ e_+(\cdot; \gamma)], \phi_\mathrm{pl} \rangle &= 0. 
    \end{cases} \label{e: pulled eigenvalue LS}
\end{align}

The system \eqref{e: pulled eigenvalue LS} has a trivial solution $(w, \beta; \gamma, \delta) = (0, 0; 0, 0)$. The linearization of the first equation with respect to $w$ is $P_0 \mathcal{A}_\mathrm{pl}^0$, which is invertible by construction, since the operator $\mathcal{A}_\mathrm{pl}^0$ has no kernel on this space, and $P_0$ projects onto its range. Hence, we may solve the first equation for $w(\beta; \gamma, \delta)$ as a function of the other parameters via the implicit function theorem. In fact, since the equation is also linear in $\beta$, this solution must by linear in $\beta$, so we write $w(\beta; \gamma, \delta) = \beta \tilde{w}(\gamma, \delta)$ for some $\tilde{w}(\gamma, \delta) \in H^2_{0, \eta}$. Inserting this into the second equation, we find the desired reduced equation
\begin{align*}
    E(\gamma, \delta) = 0,
\end{align*}
where
\begin{align*}
    E(\gamma, \delta) := \langle [d_1 \partial_x^2 + L_2(\upl, \partial_x \upl, \gamma^2; \delta) \partial_x + L_1 (\upl, \partial_x \upl, \gamma^2; \delta) - \gamma^2 ] [\tilde{w}(\gamma,\delta) +  \chi_+ e_+(\cdot; \gamma)], \phi_\mathrm{pl} \rangle .
\end{align*}
\end{proof}

\begin{proof}[Proof of Proposition \ref{p: pulled stability away from transition}]
    It follows from Lemma \ref{l: pme properties pulled} and Proposition \ref{p: pulled evans function} that $E(0,0) \neq 0$. Since $E$ is continuous in both arguments, $E(\gamma, \delta) \neq 0$ for $\gamma, \delta$ sufficiently small, so there are no eigenvalues in a neighborhood of the origin. Large eigenvalues are excluded by Proposition \ref{p: large lambda stability}, and eigenvalues away from the origin in a bounded region may be handled by regular perturbation arguments, as in Proposition \ref{p: pushed stability away from transition}. 
\end{proof}

\subsection{Proof of Theorem \ref{t: main}}
\begin{proof}
    By Corollaries \ref{c: pulled persistence near transition} and \ref{c: pulled marginal stability near transition}, there exists $\eps > 0$ and $\delta_0 > 0$, such that for $|\delta| < \delta_0$ and $d_1 \in (d_1^*(\delta) - \eps, d_1^*(\delta)$, we find marginally stable pulled fronts bifurcating from the pushed-pulled transition at $d_1^*(\delta)$. Now we can apply Proposition \ref{p: pulled persistence} to find pulled fronts, marginally stable by Proposition \ref{p: pulled stability away from transition}, for $d_1 \in [\underline{d}, d_1^*(\delta) - \eps]$ for all $|\delta| < \tilde{\delta_1}$, where $\tilde{\delta}_1$ can be chosen independent of $d_1$, since we are focused on a compact region which is uniformly away from the pushed-pulled transition. Similarly, by Proposition \ref{p: pushed persistence near transition} we find marginally stable pushed fronts bifurcating from the pushed-pulled transition for all $|\delta| < \delta_0$ and $d_1 \in (d_1^*(\delta), d_1^*(\delta) + \eps)$. Then, by Propositions \ref{p: pushed persistence} and \ref{p: pushed stability away from transition}, we find marginally stable pushed fronts for $d_1 \in [d_1^*(\delta) + \eps, \overline{d}]$ for all $|\delta| < \tilde{\delta_2}$, where $\tilde{\delta}_2$ can be chosen independent of $d_1$. By Proposition \ref{p: generic transition}, the pushed-pulled transition is generic in the sense of \cite{avery22}, as desired. 
\end{proof}

\section{Numerical Continuation Results }\label{sec:numerics}
We obtained expansions for the speed of pushed fronts and the location of the pushed-to-pulled transition as one perturbs away from the the porous medium limit in Theorem~\ref{t: main}.  In this section, we compare these estimates to values obtained via a numerical approximation.  We employ a numerical continuation routine recently proposed by the authors in \cite{avery22}.  The method approximates traveling front solutions by solving a boundary value problem making use of a far-field core decomposition.  The far-field portion encodes the exact decay rate of the front in the linearization near the unstable state while the core is a localized portion with stronger decay rates that captures the nonlinear corrections to the front profile.  As explained in \cite{avery22}, the method can approximate pulled fronts, pushed fronts, and the transition with errors that decrease exponentially in the domain size (as opposed to the algebraic errors for pulled fronts that arise if one were to solve a boundary value problem with, for example, Dirichlet boundary conditions on either side of the interval).  Since the asymptotic decay rates are explicitly included in the decomposition, the transition between pushed and pulled fronts can be located by finding parameters values for which this decay rate is purely exponential, i.e. when the parameter $a=0$ in (\ref{e: pulled ansatz}).  For more details see \cite{avery22}.

 Theorem~\ref{t: main} obtains first order corrections to the pushed front speed and the pushed-to-pulled transition point as $\delta$ is perturbed from zero.  In Figure~\ref{fig:numerics} we compare these predictions to quantities obtained from numerical continuation as described above.  The numerical continuations are performed using fourth-order discretizations of the Laplacian on a discretized spatial domain $[-L,L]$ with $L=20$ and $dx=0.1$.  The chemotactic term in the first equation of (\ref{eq:main}) is expanded and we use the second equation to replace the term $\chi u v_{xx}$ with $\chi u (v-u)/\sigma$.

\paragraph{Pushed front speeds.}
For pushed fronts, we find good approximations to the speeds for $\e$ near the critical point at $d_1=0.5$.  The predictions are less accurate for smaller values of $d_1$.  It turns out that the coefficient $c_{\mathrm{ps},2}$ is quite small and so one explanation of this deviation is that the $\mathcal{O}(\delta^4)$ in (\ref{e: cps2}) could have a non-trivial influence for $\delta^2$ on the order of $0.01$ to $0.1$.  One interesting feature that is observed is the non-uniform effect of increasing $\delta$ on the speed of the traveling front.  For $\e$ close to $0.5$, increasing $\delta$ leads to faster invasion speeds as compared to the porous medium limit while smaller values of $d_1$ lead to a decrease in the relative invasion speed.  

\paragraph{Pushed-to-pulled transition.}

Expansions for the pushed-to-pulled transition point are given in (\ref{e:d12}).  We compute the location of the pushed-to-pulled transition using numerical continuation techniques and compare them to the linear approximation $\e=\frac{1}{2}+d_{1,2}\delta^2$.  This linear approximation provides are remarkably good fit for $\delta^2$ less than approximately $0.53$.  At this value of $\delta^2$  the linearization of (\ref{eq:main}) has a spatial eigenvalue with algebraic multiplicity three and our numerical continuation routine is unable to continue through this resonance. 
For $\delta^2$ greater than this critical value the transition point is no longer approximately linear.  

A quadratic fit to the computed values of $d_1(\delta)$ in the range $0<\delta<6\cdot 10^{-3}$ gives  the approximation 
\begin{equation}\label{e:qfit}
d_1(\delta)\sim 0.49999976+   0.064829368124163 \delta^2+  0.0077592 \delta^4 .
\end{equation}
Comparing with the prediction, we find an error in the constant term of order $10^{-7}$ and a relative error in the linear term of order $5\cdot 10^{-5}$, consistent with discretization accuracy. 

The numerically computed pushed-pulled transition curve $d_1^*(\delta)$, shown in blue in the right panel of Figure \ref{fig:numerics}, appears strikingly close to linear in $\delta^2$ for $\delta^2 \leq 0.5$. To investigate whether the curve is truly linear, we computed the local slope of the curve as $\delta^2$ ranges from $0.1$ to $0.4$, and found that the slope changed by approximately three percent from $\delta^2 = 0.1$ to $\delta^2 = 0.4$. This change was robust to decreasing $dx$ and increasing $L$, suggesting that the curve $d_1^*(\delta)$ is genuinely nonlinear in $\delta^2$, but with a very small leading order nonlinear term. Note that the coefficient of the correction $\delta^4$ in \eqref{e:qfit} is small but does not vanish, also indicating that the apparent linear dependence of $d_1$ on $\delta^2$ is a good approximation yet not exact according to the numerical data.

\begin{figure}
    \centering
     \subfigure{\includegraphics[width=0.43\textwidth]{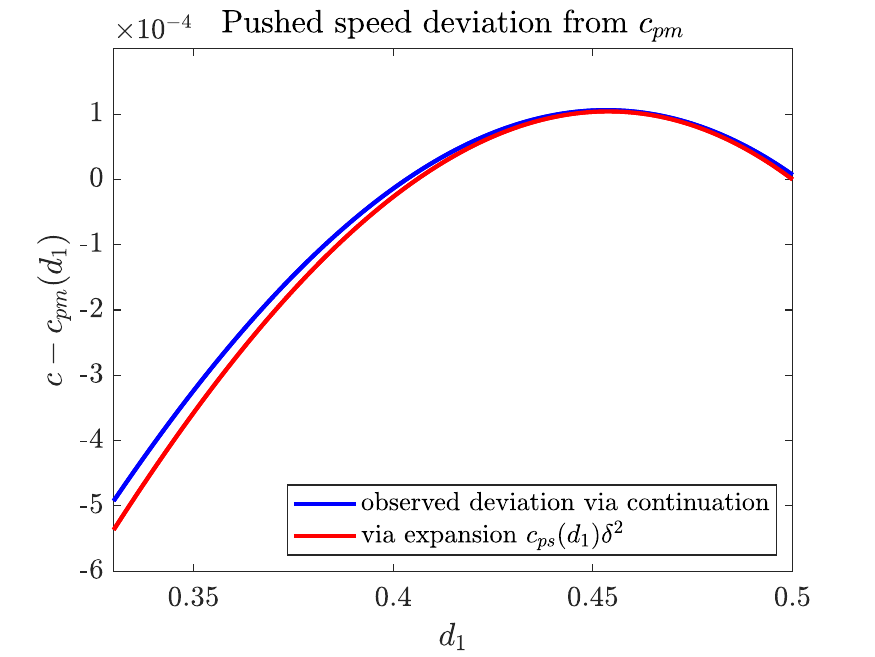}}
 \subfigure{\includegraphics[width=0.43\textwidth]{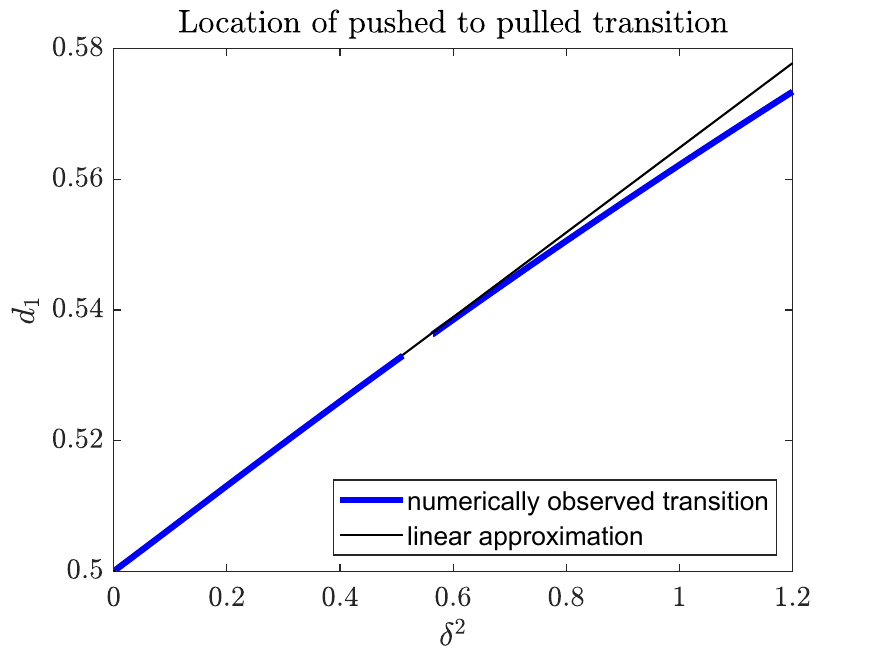}}
   \caption{On the left, we compare pushed front speeds with $\delta^2=0.1$ with the expansion obtained in (\ref{e: cps2}). Specifically, we plot the deviation of the pushed speed from the porous medium limit speed $c_{pm}(\e)$ determined by numerical continuation using the approach described in \cite{avery22} and via the expansion (\ref{e: cps2}).  On the right we show the location in $\delta^2$-$\e$ parameter space of the pushed to pulled transition determined by the same continuation method. A linear prediction for the location of the expansion curve provided by (\ref{e:d12}) is shown in black and appears to be a reliable estimate for the actual transition value beyond the limit of small $\delta$.   We remark that the gap in the data appearing near $0.53$ is due to a resonance in the eigenvalues that prevents our algorithm from continuing through that point.     }
    \label{fig:numerics}
\end{figure}

\appendix

\section{Front solutions in the porous medium limit}\label{a: pme limit}

\paragraph{Proof of Lemma~\ref{lem:PME}}

We follow the approach presented in \cite{kawasaki17}. 
Write (\ref{e: pme}) as a system of first-order equations yielding
\begin{eqnarray}
U'&=& W \nonumber \\
W'&=& -\frac{cW+W^2+U(1-U)}{\e+U}. \label{eq:reducedslow}
\end{eqnarray}
Then change coordinates via 
\[ \tilde{W}=(\e+U)W, \ \tilde{U}=U, \ \tilde{x}=\int_0^x \frac{1}{\e+U(\tau)}d\tau. \]
This transforms (\ref{eq:reducedslow}) to the system
\begin{eqnarray}
\frac{d\tilde{U}}{d\tilde{x}}&=& \tilde{W} \nonumber  \\
\frac{d\tilde{W}}{d\tilde{x}}&=& -c\tilde{W}-\tilde{U}(1-\tilde{U})(\e+\tilde{U}). \label{eq:rescaledtonagumo}
\end{eqnarray}
This is Nagumo's equation -- studied in \cite{hadeler75} -- for which the selected speed has been established to be
\be \cpm(\e)=\left\{ \begin{array}{cc} \frac{1}{\sqrt{2}}+\sqrt{2}\e & \e<\frac{1}{2} \\ 2\sqrt{\e} & \e\geq \frac{1}{2}\end{array} \right.. \ee
Thus we have established Lemma~\ref{lem:PME}

\paragraph{Front asymptotics for (\ref{e: pme})}  Lemma~\ref{l: pme properties pulled} and Lemma~\ref{l: pme properties pushed} require asymptotics for the front profile near the unstable zero state.  

\paragraph{Pushed front asymptotics} Equation (\ref{eq:rescaledtonagumo}) has a family of exact solutions lying on the quadratic curve $\tilde{W}=-\alpha \tilde{U}\left(1-\tilde{U}\right)$.  To compute for $\alpha$ and $c$, one can substitute and obtain the invariance condition 
\be \alpha^2-2\alpha c +\e-2\alpha^2\tilde{U}+\tilde{U}=0\label{eq:alphadet} \ee
from which 
$\alpha=\sqrt{\frac{1}{2}}$ and $c$ satisfying $\alpha^2-c\alpha+\e=0$; or $c=\frac{1}{\sqrt{2}}+\sqrt{2}\e$.  These heteroclinics correspond to traveling front solutions of (\ref{eq:rescaledtonagumo}) which have the explicit form,
\be \tilde{U}(\tilde{x})= \frac{e^{-\frac{\tilde{x}}{\sqrt{2}}}}{1+e^{-\frac{\tilde{x}}{\sqrt{2}}}}. \label{eq:explicitsoln} \ee
\begin{rmk} Alternatively, we notice that for these solutions 
\[
U'=W=\frac{\tilde{W}}{d_1+U}=-\frac{U(1-U)}{\sqrt{2}(d_1+U)}, 
\]
which gives the explicit expression for the inverse $\psi(y)$ 
\begin{equation}\label{e:fexpl}
    \psi(y)=\sqrt{2} ((1 + d_1) \log(1 - y) - d_1 \log(y)),\qquad \psi(U(x))=x.
\end{equation}
\end{rmk}
It is important to note that only when $\e<\frac{1}{2}$ is this front a selected, pushed front.  For $\e>\frac{1}{2}$ the front has weak exponential decay and belongs to the class of super-critical fronts which are not selected by compactly supported initial data.  When $\e<\frac{1}{2}$ the selected pushed front has decay rate 
\be \tilde{U}(\tilde{x})\sim C e^{-\frac{\tilde{x}}{\sqrt{2}}}. \ee
Note that $\tilde{x}\sim \frac{x}{\e}$ as $x\to\infty$ from which the decay rate in Lemma~\ref{l: pme properties pushed} is obtained.

\paragraph{Pulled front asymptotics}
When $\e\geq\frac{1}{2}$ the invasion fronts are pulled.  Our primary goal in this section is to verify the expansion of the front in the leading edge stated in (\ref{e:pmeanonzero}) where it is claimed that the  coefficient $a$ is positive.  Validation of this will be accomplished using a comparison argument after two changes of coordinates to simplify the analysis.  The first change of coordinates transforms (\ref{eq:rescaledtonagumo}) to projective coordinates.  These coordinates will be employed  to distinguish pure exponential decay (with $a=0$)  from weak exponential decay due to the algebraic pre-factor (with $a\neq 0$).

To begin, let $\eta=\frac{\tilde{W}}{\tilde{U}}$ after which (\ref{eq:rescaledtonagumo}) is transformed to (with $c = \clin =2\sqrt{\e}$) 
\begin{eqnarray} 
\frac{d\tilde{U}}{d\tilde{x}} &=& \eta \tilde{U} \nonumber \\ 
\frac{d\eta}{d\tilde{x}}&=& -\clin \eta -\eta^2- (1-\tilde{U})\left(\e+\tilde{U}\right) \label{eq:projective}
\end{eqnarray}
This system of equations has a fixed point at $(\tilde{U},\eta)=(0,-\sqrt{\e})$.  This fixed point is non-hyperbolic with one zero eigenvalue and one negative eigenvalue $-\sqrt{\e}$.  The center manifold can be taken to be the $\eta$ axis and there exists a one dimensional stable manifold tangent to the stable eigenvector $(\sqrt{\e},1-\e)$.

It is more convenient to re-scale (\ref{eq:projective}) so that $\eta=\sqrt{\e}z$. Also rescaling the independent variable by a factor of $\sqrt{d_1}$, and recalling that $\clin = 2 \sqrt{d_1}$, we obtain the system
\begin{eqnarray} 
\tilde{U}'&=& z \tilde{U} \nonumber  \\
z'&=& -(z+1)^2 +\tilde{U}-\frac{1}{\e}\tilde{U}(1-\tilde{U}). \label{eq:zU}
\end{eqnarray}
This system has fixed points at $(0,-1)$ and $(1,0)$ for all $\e>0$.  Due to the explicit front solution (\ref{eq:explicitsoln}) we know that the stable manifold of $(0,-1)$ intersects the unstable manifold of $(1,0)$ at $\e=\frac{1}{2}$.  The fixed point $(0,-1)$ has stable eigenvalue $-1$ with eigenvector $\left(1,\frac{1}{\e}-1\right)$. The fixed point at $(1,0)$ is hyperbolic with eigenvalues $-1\pm\sqrt{2+\frac{1}{\e}}$.  The unstable eigenvector is proportional to $\left(-1,1-\sqrt{2+\frac{1}{\e}}\right)$. 

Let $h_s(\tilde{U},\e)$ be the graph of the stable manifold of $(0,-1)$ and let $h_u(\tilde{U},\e)$ be the graph of the unstable manifold of $(1,0)$.  Based upon the eigenvectors, we know that for $0<U$ and sufficiently small that $h_s(\tilde{U},\e)$ is a monotone decreasing function of $\e$.  Similarly, for $0<U<1$ sufficiently close to $1$ it holds that $h_u(\tilde{U},\e)$ is monotone increasing in $\e$.

Let $\phi(\tilde{U})$ be the graph of the heteroclinic orbit connecting $(0,-1)$ and $(1,0)$ with $\e=\frac{1}{2}$.   Suppose for the sake of contradiction we assume the existence of a heteroclinic connection for $\e\neq \frac{1}{2}$.  Let $\psi(\tilde{U},\e)$ be the graph of this heteroclinic. Fix $\e>\frac{1}{2}$.  Then it follows from the properties of the stable and unstable manifolds of these fixed points that there exists a $\tilde{U}_1$ sufficiently small so that 
\[ \phi(\tilde{U}_1)>\psi(\tilde{U}_1,\e),  \]
and a $\tilde{U}_2$ close to $1$ such that 
\[ \phi(\tilde{U}_2)<\psi(\tilde{U}_2,\e).  \]
Then, in order for there to exist a heteroclinic orbit for $\e$ it must be that there exists a $\tilde{U}_1<\tilde{U}_*<\tilde{U}_2$ such that 
\be \phi(\tilde{U}_*)=\psi(\tilde{U}_*,\e)=z^*, \quad \phi'(\tilde{U}_*)<\psi'(\tilde{U}_*,\e). \label{eq:derivativecond} \ee
Compute the derivative 
\be \frac{dz}{dU}=-\frac{(z+1)^2}{zU}+\frac{1}{z}-\frac{1}{\e}\frac{1-U}{z}. \label{eq:phi'} \ee
Then the derivative condition in (\ref{eq:derivativecond}) requires  
\[ -\frac{1}{\e z^*}(1-U^*)> -\frac{2}{ z^*}(1-U^*), \]
which, after recalling that $z^*<0$, only holds if $\e<\frac{1}{2}$.  

As a consequence there can exist no heteroclinic connection for $\e>\frac{1}{2}$.  Tracking the unstable manifold of $(1,0)$ forward in $\tilde{x}$ we see that it can not cross the curve $\phi(\tilde{U})$. This means that $z(\tilde{x})>-1$ and after untangling the changes of coordinates we have that 
\[ \frac{\tilde{U}'}{\tilde{U}}\sim -\sqrt{\e}+\frac{\sqrt{d_1}}{\tilde{x}+\kappa}, \]
for some $\kappa>0$ from which 
we obtain that $\tilde{U}(\tilde{x})\sim (a\tilde{x}+b)e^{-\sqrt{d_1}\tilde{x}}$ with $a>0$ as claimed.    

\section{Spectral stability for large $|\lambda|$}\label{s: large lambda stability}

Recall the formulation \eqref{e: eigenvalue sys} of the eigenvalue problem as a first-order system,
\begin{align}
    \tilde{u}_x &= \tilde{w} \nonumber \\
    \tilde{w}_x &= -\frac{1}{d_1} D \Gamma(U, W, H, Z) (\tilde{u}, \tilde{w}, \tilde{h}, \tilde{z})^T + \frac{1}{d_1} \lambda \tilde{u} \nonumber \\
    \delta \tilde{h}_x &= \tilde{z} \nonumber \\
    \delta \tilde{z}_x &= \tilde{h} + \frac{1}{d_1} D\Gamma(U,W,H,Z) (\tilde{u}, \tilde{w}, \tilde{h}, \tilde{z})^T - \frac{1}{d_1} \lambda \tilde{u}, \label{e: app slow sys}
\end{align}
where $(U(x),W(x),H(x),Z(x))$ denotes any heteroclinic solution to \eqref{e: tw sys} satisfying $0 < U < 1$. Defining the rescaled (spatial) time $y = \frac{x}{\delta}$, we find the \emph{fast system}, 
\begin{align}
    \tilde{u}_y &= \delta \tilde{w} \nonumber \\
    \tilde{w}_y &= - \frac{\delta}{d_1} D(U(\delta y), W(\delta y), H(\delta y), Z(\delta y)) + \frac{\delta}{d_1} \lambda \tilde{u} \nonumber \\
    \tilde{h}_y &= \tilde{z} \nonumber \\ 
    \tilde{z}_y &= \tilde{h} + \frac{1}{d_1} D\Gamma(U,W,H,Z) (\tilde{u}, \tilde{w}, \tilde{h}, \tilde{z})^T - \frac{1}{d_1} \lambda \tilde{u}, \label{e: app fast sys}
\end{align}
which is equivalent to \eqref{e: app slow sys} for $\delta > 0$. To recognize the leading order dynamics for large $|\lambda|$, we set $\lambda = \frac{1}{\gamma^2}$, rescale time to $\xi = \frac{y}{|\gamma|}$, and set $\hat{w} = |\gamma| \tilde{w}$, finding the equivalent system
\begin{align}
    \tilde{u}_\xi &= \delta \hat{w} \nonumber \\
    \hat{w}_\xi &= \frac{\delta}{d_1} \frac{|\gamma|^2}{\gamma^2} \tilde{u} - \delta \frac{|\gamma|^2}{d_1} D\Gamma (U_2, W_2, H_2, Z_2) \cdot \left(\tilde{u}, \frac{\hat{w}}{|\gamma|}, \tilde{h}, \tilde{z}\right)^T \nonumber \\
    \tilde{h}_\xi &= |\gamma| \tilde{z} \nonumber \\
    \tilde{z}_\xi &= |\gamma| \tilde{h} - \frac{|\gamma|}{\gamma^2} \frac{1}{d_1} \tilde{u} + \frac{|\gamma|}{d_1} D\Gamma(U_2, W_2, H_2, Z_2) \left(\tilde{u},  \frac{\hat{w}}{|\gamma|}, \tilde{h}, \tilde{z}\right)^T,
\end{align}
where $(U_2(\xi), W_2(\xi), H_2 (\xi), Z_2(\xi)) = (U(\delta |\gamma| \xi), W(\delta |\gamma| \xi), H(\delta |\gamma| \xi), Z(\delta |\gamma| \xi))$ are slowly varying if $|\delta \gamma|$ is small. Explicitly evaluating $D\Gamma$ and rescaling $(\tilde{h}, \tilde{z}) = \frac{1}{|\gamma|} \left( \hat{h}, \hat{z} \right)$, we find
\begin{align}
    \tilde{u}_\xi &= \delta \hat{w} \nonumber \\
    \tilde{w}_\xi &= \frac{\delta}{d_1} \frac{|\gamma|^2}{\gamma^2} \tilde{u} - \frac{\delta |\gamma|^2}{d_1} \left[ \tilde{u} (H_2 + 1 - 2U_2) + \frac{\hat{w}}{|\gamma|} (c + 2 W_2 + \delta Z_2) + U_2 \frac{\hat{h}}{|\gamma|} + \delta W_2 \frac{\hat{z}}{|\gamma|}\right] \nonumber \\
    \hat{h}_\xi &= |\gamma| \hat{z} \nonumber \\
    \hat{z}_\xi &= |\gamma| \hat{h} - \frac{|\gamma|^2}{\gamma^2} \frac{1}{d_1} \tilde{u} + \frac{|\gamma|^2}{d_1} \left[ \tilde{u} (H_2 + 1 - 2U_2) + \frac{\hat{w}}{|\gamma|} (c + 2 W_2 + \delta Z_2) + U_2 \frac{\hat{h}}{|\gamma|} + \delta W_2 \frac{\hat{z}}{|\gamma|}\right]. 
\end{align}
We make one more rescaling, defining $(\check{h}, \check{z}) = |\gamma \delta|^{1/2} (\hat{h}, \hat{z})$, so that the system becomes
\begin{align}
\tilde{u}_\xi &= \delta \hat{w}, \nonumber \\
\hat{w}_\xi &= \frac{\delta}{d_1} \frac{|\gamma|^2}{\gamma^2} \tilde{u} - \frac{|\delta \gamma|^{1/2}}{d_1} U_2 \check{h} + \tilde{\Gamma}_1 (U_2, W_2, H_2, Z_2; \gamma, \delta) \cdot (\tilde{u}, \hat{w}, \check{h}, \check{z}), \nonumber \\
\check{h}_\xi &= |\gamma| \check{z} \nonumber \\
\check{z}_\xi &= |\gamma| \left(1 + \frac{U_2}{d_1} \right) \check{h} - \frac{|\gamma|^2}{\gamma^2} \frac{|\gamma \delta|^{1/2}}{d_1} \tilde{u} + \tilde{\Gamma}_2 (U_2, W_2, H_2, Z_2; \gamma, \delta) \cdot (\tilde{u}, \hat{w}, \check{h}, \check{z}), \label{e: app finally rescaled}
\end{align}
where
\begin{align*}
    \tilde{\Gamma}_1(U_2, W_2, H_2, Z_2; \gamma, \delta) &= \begin{pmatrix}
        - \frac{\delta |\gamma|^2}{d_1} (H_2 + 1 - 2 U_2) & - \frac{|\delta \gamma|}{d_1} (c + 2 W_2 + \delta Z_2) & 0 & -\delta |\delta \gamma|^{1/2} W_2
    \end{pmatrix}, \\
    \tilde{\Gamma}_2 (U_2, W_2, H_2, Z_2; \gamma, \delta) &= \begin{pmatrix}
        \frac{|\gamma|^2 |\delta \gamma|^{1/2}}{d_1} (H_2 + 1 - 2 U_2) & \frac{|\gamma| |\delta \gamma|^{1/2}}{d_1} (c + 2 W_2 + \delta Z_2) & 0 & \frac{\delta |\gamma|}{d_1} W_2
    \end{pmatrix}. 
\end{align*}
Note that the principal terms, written explicitly in \eqref{e: app finally rescaled}, all have coefficients on the order $\mathrm{O}(|\gamma|, |\delta|, |\delta \gamma|^{1/2})$, while all terms in $\tilde{\Gamma}_{1/2}$ are higher order. Our goal is now to show that there exists $r_1 > 0$ and $\phi_0 \in (\frac{\pi}{4}, \frac{\pi}{2})$ such that \eqref{e: app finally rescaled} admits no bounded solutions provided
\begin{align}
    |\delta| < r_1, \quad |\gamma| < r_1, \quad |\mathrm{Arg} \, \gamma| < \phi_0. \label{e: app conditions}
\end{align}
We set $\frac{|\gamma|^2}{\gamma^2} = e^{i \theta}$, and first consider the coordinate chart $\delta = \delta_1 |\gamma|$ on the $(\delta, |\gamma|)$-plane, in which we find
\begin{align}
    \tilde{u}_\xi &= \delta_1 |\gamma| \hat{w} \nonumber  \\
    \hat{w}_\xi &= \frac{\delta_1}{d_1} |\gamma| e^{i \theta} \tilde{u} - \frac{|\delta_1|^{1/2}}{d_1} |\gamma| U_2 \check{h} + \mathrm{O}(|\gamma|^2) \cdot (\tilde{u}, \hat{w}, \check{h}, \check{z})^T \nonumber \\
    \check{h}_\xi &= |\gamma| \check{z} \nonumber \\
    \check{z}_\xi &= |\gamma| \left(1 + \frac{U_2}{d_1} \right) \check{h} - e^{i \theta} |\delta_1|^{1/2} \frac{|\gamma|}{d_1} \tilde{u} + \mathrm{O}(|\gamma|^2) \cdot (\tilde{u}, \hat{w}, \check{h}, \check{z})^T. \label{e: app rescaled 1}
\end{align}
Rescaling time to remove the Euler multiplier $|\gamma|$, we find that the eigenvalues of the resulting leading order system are given as roots of the characteristic polynomial
\begin{align}
    0 = \det \begin{pmatrix}
        \frac{\delta_1^2}{d_1} e^{i \theta}-\nu^2 & -\frac{|\delta_1|^{3/2}}{d_1} U_2 \\
        -e^{i\theta} \frac{|\delta_1|^{1/2}}{d_1} & 1 + \frac{U_2}{d_1} - \nu^2  
    \end{pmatrix} =: \det (M - \nu^2 I). \label{e: app characteristic}
\end{align}
\begin{lemma}\label{l: app rescaling 1}
    Fix $d_1 > 0$. For any $\delta_1 > 0$, $0 \leq U_2 \leq 1$, and $\theta$ with $|\theta| < \frac{3 \pi}{4}$, \eqref{e: app characteristic} has no roots $\nu$ which are purely imaginary. 
\end{lemma}
\begin{proof}
    Suppose we have a root $\nu = ik$ of \eqref{e: app characteristic} which is purely imaginary. After some rearranging, we find
    \begin{align}
        e^{i \theta} = \frac{-k^2 \left( 1 + \frac{U_2}{d_1} + k^2 \right)}{\frac{\delta_1^2}{d_1^2} (1+k^2)},
    \end{align}
    which implies in particular that we must have $\theta = \pi$.  
\end{proof}
\begin{corollary}\label{c: delta 1 bounded}
    Fix $d_1 > 0$. For any $\delta_1^* >  0$, there exists $\gamma_1^* > 0$ such that the system \eqref{e: app rescaled 1} has no bounded solutions with $\delta = \delta_1 |\gamma|$, $\delta_1^* \leq \delta_1 \leq \frac{1}{\delta_1^*}$, $|\gamma| \leq \gamma_1^*$, and $|\mathrm{Arg} \, \gamma| \leq \frac{3 \pi}{8}$. 
\end{corollary}
\begin{proof}
    Since the linear system \eqref{e: app rescaled 1} has slowly varying coefficients, the hyperbolicity captured in Lemma \ref{l: app rescaling 1} implies that \eqref{e: app rescaled 1} admits exponential dichotomies in the desired parameter regime \cite[Lemma 2.3]{Sakamoto}, and the existence of exponential dichotomies rules out the possibility of bounded solutions. 
\end{proof}
\begin{lemma}\label{l: delta 1 small}
    Fix $d_1 > 0$. There exist $\delta_1^\dag, \gamma_1^\dag > 0$ such that \eqref{e: app rescaled 1} admits no bounded solutions with $\delta = \delta_1 |\gamma|, |\delta_1| < \delta_1^\dag, |\gamma| \leq \gamma_1^\dag$, and $|\mathrm{Arg} \, \gamma| \leq \frac{3 \pi}{8}$. 
\end{lemma}
\begin{proof}
    After removing the Euler multiplier $|\gamma|$ and by coupling to the corresponding rescaled version of the existence problem, as in Section \ref{s: regularization}, we find that for $\delta_1, \gamma_1$ small all bounded solutions of \eqref{e: app rescaled 1} lie on a normally hyperbolic slow manifold, with leading order expansion
    \begin{align*}
        \check{h} = e^{i \theta} \frac{|\delta_1|^{1/2}}{d_1} \frac{\tilde{u}}{1 + \frac{U_2}{d_1}}, \quad \check{z} = \mathrm{O}(|\delta_1|^{1/2})
    \end{align*}
    Using \eqref{e: app rescaled 1}, we therefore find the reduced flow on the slow manifold is governed to leading order by by
    \begin{align}
        \tilde{u}_{\xi \xi} + \frac{1}{d_1} e^{i \theta} \left( \frac{U_2}{d_1 + U_2} - 1 \right) \tilde{u} = 0. 
    \end{align}
    For fixed $0 \leq U_2 \leq 1$ and $|\theta| \leq \frac{3 \pi}{4}$, the corresponding first-order system is hyperbolic. Since the coefficients are again slowly varying, we conclude the existence of exponential dichotomies, and hence non-existence of bounded solutions, again by \cite[Lemma 2.3]{Sakamoto}.  
\end{proof}

Combining Corollary \ref{c: delta 1 bounded} and \ref{l: delta 1 small}, we have excluded bounded solutions to \eqref{e: app finally rescaled} with $\delta = \delta_1 |\gamma|, |\gamma| \leq \max (\gamma_1^*, \gamma_1^\dag)$, and $|\mathrm{Arg} \gamma| \leq \frac{3 \pi}{8}$. It only remains to exclude the regime $|\gamma| = \gamma_1 \delta$, with $\gamma_1 > 0$ small. This argument is completely analogous to the proof of Lemma \ref{l: delta 1 small}, and so we conclude that there exists $r_1 > 0$ such that \eqref{e: app finally rescaled} admits no bounded solutions satisfying \eqref{e: app conditions} with $\phi_0 = \frac{3 \pi}{8}$. For $\delta, \gamma$ nonzero, the existence of bounded solutions to \eqref{e: app finally rescaled} is equivalent to existence of bounded solutions of the original eigenvalue problem \ref{e: app slow sys}, so we have proved Proposition \ref{p: large lambda stability}.


\bibliographystyle{abbrv}
\bibliography{KSbib}

\end{document}